\documentclass[a4paper,12pt,reqno]{amsart}
\usepackage{amssymb}
\usepackage{amsmath}
\usepackage{amsthm}
\usepackage{graphics}
\usepackage{verbatim}
\usepackage{rcs}
\usepackage{color}
\usepackage{comment}

\RCS $Id: HDL-nonconvex.tex,v 1.79 2017/03/18 12:36:09 nisikawa Exp $

\DeclareFontFamily{U}{rsfs}{\skewchar\font"7F}
\DeclareFontShape{U}{rsfs}{m}{n}{
	<-6> rsfs5
	<6-8> rsfs7
	<8-> rsfs10
	}{}
\DeclareMathAlphabet{\mathscr}{U}{rsfs}{m}{n}

\theoremstyle{plain}
\newtheorem{theorem}{Theorem}[section]
\newtheorem{prop}[theorem]{Proposition}
\newtheorem{lemma}[theorem]{Lemma}

\theoremstyle{definition}

\theoremstyle{remark}
\newtheorem{ass}{Assumption}[section]
\newtheorem{example}{Example}[section]
\newtheorem{remark}{Remark}[section]

\def\real{{\mathbb R}}
\def\natural{{\mathbb N}}
\def\integer{{\mathbb Z}}

\def\dive{\mathop{\mathrm{div}}\nolimits}

\numberwithin{equation}{section}

\begin{document}
\title[Hydrodynamic limit for interface model with non-convex potential]{Hydrodynamic limit for the Ginzburg-Landau $\nabla\phi$ interface model with non-convex potential}
\author[J.-D.~Deuschel \and T.~Nishikawa \and
Y.~Vignaud]
{Jean-Dominique Deuschel \and Takao Nishikawa \and Yvon Vignaud}
\address{J.-D.Deuschel: Institut f\"ur Mathematik,
Technische Universit\"at Berlin, Berlin, Germany \newline
\indent {\it E-mail address}: {\tt deuschel@math.tu-berlin.de} \newline \quad \newline
\indent T.~Nishikawa: Department of Mathematics, College of Science and Technology,
Nihon University, Tokyo, Japan \newline
\indent {\it E-mail address}: {\tt nisikawa@math.cst.nihon-u.ac.jp} \newline \quad \newline
\indent Y.~Vignaud: Lyc\'ee Jean Jaur\`es, Argenteuil, France
}
\maketitle
%
%

\begin{abstract}
Hydrodynamic limit for the Ginzburg-Landau $\nabla\phi$ interface model was established in \cite{N03} under the Dirichlet boundary conditions.
This paper studies the similar problem,
but with non-convex potentials. Because of the lack of
strict convexity, a lot of difficulties arise, especially,
on the identification of equilibrium states.
We give a proof of the
equivalence between the stationarity and the Gibbs property
under quite general settings, and as its conclusion, we complete the identification of equilibrium states
under the high temparature regime in \cite{CD08}.
We also establish some uniform estimates for variances
of extremal Gibbs measures under quite general settings. 
\end{abstract}

\section{Introduction}
\label{sec-model}
We consider the large scale hydrodynamic behavior of the                                           
the Ginzburg-Landau
$\nabla\phi$ interface model. This is an effective interface model, describing the
stochastic dynamic of the separation
of two distinct phases.

The position of the interface is described by height variables
$\phi=\{\phi(x)\in\real;\,x\in \Gamma\}$ measured from a fixed $d$-dimensional
discrete hyperplane $\Gamma$. 
Here, we will take $\Gamma=\Gamma_N:=(\integer/N\integer)^d$ when we consider the system on a discretized torus with the periodic
boundary condition, or  $\Gamma=D_N\subseteq\integer^d$ when 
we consider the system on the domain $\Gamma$ with Dirichlet boundary condition.  $D_N$ is a
microscopic domain corresponding to a given macroscopic domain $D\subset\real^d$
which is bounded and has a smooth boundary. 
See Section~2 for the precise definition.

The corresponding Hamiltonian  $H(\phi)$ on $\Gamma$
for given 
height variable $\phi$ is of the form 
\[
	H(\phi)=\frac{1}{2}\sum_{\begin{subarray}{c}
	x,y\in\Gamma,\\ |x-y|=1\end{subarray}}V(\phi(x)-\phi(y))
	+\sum_{
	\begin{subarray}{c}x\in\Gamma,y\in\integer^d\smallsetminus\Gamma,
	\\|x-y|=1\end{subarray}}V(\phi(x)-\phi(y)),
\]
with a symmetric function $V\in C^2(\real)$. 
The Langevin equation
associated with $H$ is given by 
\begin{equation*}
	d\phi_t(x)=-U_x(\phi_t)\,dt
	+dw_t(x),\quad x\in\Gamma,
\end{equation*}
 where $U_x(\phi)$ in the drift term is defined by
\begin{equation*}
	U_x(\phi):=\frac{\partial H}{\partial\phi(x)}(\phi)
	\equiv \sum_{y\in\integer^d;\,|x-y|=1}V'(\phi(x)-\phi(y))
\end{equation*}
 and $\{w_t(x);\,x\in \Gamma\}$ is a family of independent copies 
of the one dimensional standard Brownian motion.

The aim of this paper investigate and identify the hydrodynamic limit of $\phi_t$
at diffusive scaling, that is, $N^2$ for time while $N$ for space.
In the case of a strictly convex potential $V$ for which 
there exist two constants $c_{+},c_{-}>0$ such that
\begin{equation}\label{cond-convex}
	c_{-}\le V''(\eta)\le c_{+},\quad \eta\in\real.
\end{equation}
 the hydrodynamic limit has been established for periodic lattice $\Gamma_N$ in
 \cite{FS97} and for  discretized domain $D_N$ with Dirichlet boundary conditions in \cite{N03}. 
%
%
%
%
In particular, the corresponding  macroscopic motion is identified as the solution
of the nonlinear partial differential equation
\begin{equation*}
	\frac{\partial h}{\partial t}=\dive\left\{(\nabla\sigma)(\nabla h(t,\theta)\right\},\quad \theta\in D,\,t>0,
\end{equation*}
where the surface tension $\sigma:\real^d\to\real$ is defined via thermodynamic limit.

In these results, the condition \eqref{cond-convex} plays an essential role
in the analysis for the stochastic dynamics $\phi_t$,
especially, in the identification of equilibrium states
and the establishment of the strict convexity of $\sigma$.
The our aim in this paper is to prove the hydrodynamic limit without the strict
convexity assumption \eqref{cond-convex},
see Assumptions~\ref{ass-pot}, \ref{ass-gibbs-uniq} and \ref{ass-surfacetension} for details.

Our motivation comes from recent results
in \cite{CD08} and \cite{CDM08} 
where both strict convexity of the surface tension and identification
of the extremal gradient Gibbs measures hold, 
for non-convex potential $V$ at sufficiently high temperature.


In the case
of the dynamics on the torus $\Gamma_N$, the limit follows quite simply
from  additional
estimates.
However, for the dynamics on the discretized domain $D_N$ with 
the Dirichlet boundary condition, the derivation is much harder, since we can not use the relative entropy and
entropy production. The main step then is to characterize the 
set of stationary measures for the gradient field associated with the infinite system of SDEs,
 which is essentially used
in order to establish local equilibrium as in \cite{N03} without using
the relative entropy and the entropy production.

In case of strictly convex $V$, the structure of the translation invariant stationary measures 
is completely identified by \cite{FS97},
its proof relying the assumption \eqref{cond-convex}.
To complete our  proof of the hydrodynamic limit in the non-convex case,
we need to identify the class of translation invariant stationary measures 
as the class of Gibbs distributions.

This subject has been  intensively studied in the literature, cf. \cite{HS77} for
stochastic Ising models,
\cite{HS81} for the diffusion process on the infinite dimensional torus $(\real/\integer)^{\integer^d}$,
  \cite{F82a} for the diffusion process on $\real^{\integer^d}$,
 \cite{BRW04} for the diffusion process on the infinite product
 $M^{\integer^d}$ with a Riemannian manifold $M$ with positive curvature.
In this paper we  show the similar result, adapting the argument
of \cite{F82a} to gradient Gibbs distributions. The main challenge here is the lack of ellipticity 
of the gradient dynamic, 
see Section~3 and 5 for details.

An alternative derivation of the hydrodynamic limit for the Ginzburg-Landau model
based on a two scale argument has been proposed
by \cite{GOVW09} and \cite{F12}.
Unlike our proof, relying on the assumption on the uniqueness of the extremal gradient Gibbs distribution, the two scale argument uses logarithmic Sobolev
inequalities. However, this approach seems restricted to the one-dimensional case in \cite{GOVW09}, respectively  strict convexity assumption
for the potential \eqref{cond-convex} in \cite{F12}.


Before closing this section, let us give briefly the organization
of this paper. In Section~2, we formulate our problem more precisely,
and state the main result. 
In Section~3, we present some properties of 
translation invariant stationary measures,
especially, the relationship between stationarity
and the Gibbs property, and some uniform estimates
for their variances. 
Note that results in this section hold under the quite general
Assumption~\ref{ass-pot}.
In Section~4, after establishing a priori bounds
for stochastic dynamics and summarize properties
of the surface tension, we derive the macroscopic
equation from the stochastic dynamics. Here, we rely
quite explicitly on the further 
Assumptions~\ref{ass-gibbs-uniq}
and \ref{ass-surfacetension}.
In Section 5, we give a proof of Theorem~\ref{th2.1}, presented
at Section~3.

\section{Model and main result}
\subsection{Model}
Let $D$ be a bounded domain in $\real^d$ with a Lipschitz boundary.
For convenience, let $D$ contain the origin of $\real^d$.
Let $D_N$ be the discretized microscopic domain corresponding to $D$
in the sense that
\[D_N=\{x\in\integer^d;\,B(x/N,5/N)\subset D\},\]
where $B(\alpha,l)$ stands for the hypercube in $\real^d$
with center $\alpha$ and side length $l$, that is,
\[B(\alpha,l)=\prod_{i=1}^d[\alpha_i-l/2,\alpha_i+l/2).\]
On $D_N$ we consider the dynamics governed
by the following stochastic differential equations (SDEs)
\begin{equation}\label{SDE1}
	d\phi_t(x)=-U_x(\phi_t)\,dt+\sqrt{2}dw_t(x),\quad x\in D_N,
\end{equation}
with the boundary condition 
\begin{equation}\label{bc}
	\phi_t(x)=\psi^N(x),\quad x\in\integer^d\smallsetminus D_N
\end{equation}
with some $\psi^N\in\real^{\integer^d}$ and initial data $\phi_0$,
where $U_x(\phi)=\frac{\partial H}{\partial\phi(x)}(\phi)$
for $\phi\in\real^{D_N}$ and $x\in D_N$, or more generally for $\phi\in\real^{\integer^d}$ and $x\in\integer^d$.
The height variable $\psi^N$ in \eqref{bc} is defined by
\begin{equation}\label{eq2.1g}
\psi^N(x)=N^{d+1}\int_{B(x/N,1/N)}f(\theta)\,d\theta
\end{equation}
for every $x\in\integer^d$, where
$f:\real^d\to\real$ is a function belonging to $C^2_0(\real^d)$.
We note that the function $f$ describes the macroscopic boundary condition and
the height variable $\psi^N$ describes the microscopic one.

We make the following assumption on the interaction potential $V$: 
\begin{ass}\label{ass-pot}
The function $V:\real\to\real$ has the following representation:
	\[V(\eta)=V_0(\eta)+g(\eta),\quad \eta\in\real,\]
	where functions $V_0,g\in C^2(\real)$ are symmetric functions
	and satisfy
	\begin{enumerate}
		\item There exist constants $c_{+},c_{-}>0$ such that
		\[c_{-}\le V_0''(\eta)\le c_{+},\quad \eta\in\real.\]
		\item There exists a constant $C_g>0$ such that
		\[|g'(\eta)|+|g''(\eta)|\le C_{g},\quad \eta\in\real.\]
	\end{enumerate}
\end{ass}

\begin{example}
If a function $V\in C^2(\mathbb{R})$ is symmetric and
satisfies
\[
c\le V''(\eta)\le c',\quad |x|\ge M
\]
for some $c,c'>0$ and $M>0$, then the function $V$ admits the decomposition
as in Assumption~\ref{ass-pot}. 
Indeed, we can take $V_0$ as follows:
\[
	V_0(x)=\begin{cases}
		\dfrac{1}{2}V''(M)x^2-\dfrac{1}{2}V''(M)M^2+V(M)+\alpha M,& |x|\le M, \\[3mm]
		V(x)+\alpha |x|,& |x|> M,
	\end{cases}
\]
with $\alpha=V''(M)M-V'(M)$. Letting $g:=V-V_0$, that is,
\[
	g(x)=\begin{cases}
		V(x)-V(M)-\dfrac{1}{2}V''(M)x^2+\dfrac{1}{2}V''(M)M^2-\alpha M,& |x|\le M, \\[3mm]
		-\alpha |x|,& |x|> M,
	\end{cases}
\]
we can easily see that
$V_0,g\in C^2(\real)$ and they fulfill conditions (1) and (2)
in Assumption~\ref{ass-pot}.
\end{example}
Further assumptions dealing with the strict convexity of
the surface tension  and  the characterization of extremal gradient  Gibbs measures
are stated below, 
see Assumptions~\ref{ass-gibbs-uniq} and \ref{ass-surfacetension} for details.

We regard \eqref{SDE1} as the model describing the motion of microscopic
interfaces and introduce the macroscopic height variable $h^N$ as follows:
\begin{equation*}
	h^N(t,\theta)=\sum_{x\in\integer^d}N^{-1}\phi_{N^2t}(x)1_{B(x/N,1/N)}(\theta),\quad \theta\in\real^d,
\end{equation*}
where $\phi_t=\{\phi_t(x);\,x\in\integer^d\}$ being the solution
of \eqref{SDE1} with \eqref{bc}.

\subsection{Notations}\label{sec-notations}
Before stating the detail of our main result, we need to introduce
several notations.
Note that we will follow the same manner as in \cite{FS97} and \cite{N03}.

Let $(\integer^d)^*$ be the set of all directed bonds $b=(x,y),\,
x,y\in\integer^d,|x-y|=1$ in $\integer^d$.
We write $x_b=x$ and $y_b=y$ for $b=(x,y)$.
We denote the bond $(e_i ,0)$ by $e_i $ again
if it doesn't cause any confusion.
For every subset $\Lambda$ of $\integer^d$, we denote
the set of all directed bonds included $\Lambda$ and touching $\Lambda$
by $\Lambda^*$ and $\overline{\Lambda^*}$, respectively.
That is,
\begin{align*}
	\Lambda^*&:=\{b\in(\integer^d)^*;\,x_b\in\Lambda
	\text{ and }y_b\in\Lambda\},\\
	\overline{\Lambda^*}&:=\{b\in(\integer^d)^*;\,x_b\in\Lambda
	\text{ or }y_b\in\Lambda\}.
\end{align*}

For $\phi=\{\phi(x);\,x\in\integer^d\}\in\real^{\integer^d}$,
the gradient $\nabla$ is defined by
\begin{gather*}
\nabla\phi(b):=\phi(x)-\phi(y),\quad
b=(x,y)\in(\integer^d)^*.
\end{gather*}
Now, let $\mathcal{X}$ be the family of
all gradient fields $\eta\in\real^{(\integer^d)^*}$ which satisfy the plaquette
condition (2.1) in \cite{FS97}, i.e., $\mathcal{X}=\{\eta\equiv\nabla\phi;\,\phi\in \real^{\integer^d}\}$.
Let $\mathbb{L}^2_r$ be the set of all $\eta\in\real^{(\integer^d)^*}$
such that
\[|\eta|_r^2:=\sum_{b\in(\integer^d)^*}|\eta(b)|^2e^{-2r|x_b|}<\infty.\]
We denote $\mathcal{X}_r=\mathcal{X}\cap\mathbb{L}^2_r$ equipped with the
norm $|\cdot|_r$.
We introduce the dynamics $\eta_t\in \mathcal{X}$ governed by the SDEs
\begin{equation}\label{SDE2}
d\eta_t(b)=-\nabla U_\cdot(\eta_t)(b)\,dt+\sqrt{2}d\nabla w_t(b),\quad b\in
(\integer^d)^*,
\end{equation}
where $\{w_t(x);\,x\in\integer^d\}$ is the family of independent
one dimensional Brownian motions.
Since the coefficients are Lipschitz continuous in $\mathcal{X}_r$,
this equation has the unique strong solution in $\mathcal{X}_r$
for every $r>0$.
Note that $\eta_t:=\nabla\phi_t$ defined from the solution
$\phi_t$ of the SDE \eqref{SDE1} on $D_N$ satisfies \eqref{SDE2} for $b\in \overline{D_N^*}$ and boundary conditions $\eta_t(b)=\nabla\psi^N(b)$ for $b\in(\integer^d)^*\smallsetminus \overline{D_N^*}$ when letting $w_t(x)\equiv0$ for $x\in \integer^d\smallsetminus D_N$.

Since we define Gibbs measures on $\mathcal{X}$
by Dobrushin-Lanford-Ruelle (DLR, for short) equation,
we the finite volume Gibbs measure in advance.
For a finite set $\Lambda\subset\integer^d$ and fixed $\xi\in\mathcal{X}$, we define
the affine space $\mathcal{X}_{\Lambda,\xi}\subset\mathcal{X}$
by
\[\mathcal{X}_{\Lambda,\xi}=\{\eta\in\mathcal{X};\, \eta(b)=\xi(b),\,b\in(\integer^d)^*
\smallsetminus \overline{\Lambda^*}\}.\]
We define the finite volume Gibbs measure $\mu_{\Lambda,\xi}$
on $\overline{\Lambda^*}$ by
\[
\mu_{\Lambda,\xi}(d\eta)=Z^{-1}_{\Lambda,\xi}\exp\left(
-\sum_{b\in\overline{\Lambda^*}}V(\eta(b))\right)d\eta_{\overline{\Lambda^*},\xi},
\]
where $d\eta_{\overline{\Lambda^*},\xi}$ is the Lebesgue measure on $\mathcal{X}_{\overline{\Lambda^*},\xi}$
and $Z_{\Lambda,\xi}$ is the normalizing constant.

Let $\mathcal{P}(\mathcal{X})$ be the set of all probability measures on
$\mathcal{X}$ and let $\mathcal{P}_2(\mathcal{X})$ be those
$\mu\in\mathcal{P}(\mathcal{X})$ satisfying $E^\mu[|\eta(b)|^2]<\infty$
for each $b\in(\integer^d)^*$. 
The measure $\mu\in\mathcal{P}_2(\mathcal{X})$ is sometimes
called tempered. 
Let $\mathcal{G}$ be the family of translation invariant, tempered Gibbs
measures $\mu\in\mathcal{P}_2(\mathcal{X})$
introduced by \cite{FS97}, namely, the family of 
$\mu\in\mathcal{P}_2(\mathcal{X})$ satisfying the
Dobrushin-Lanford-Ruelle equation
\begin{equation}\label{DLR-eq}
\mu(\cdot|\mathscr{F}_{(\integer^d)^*\smallsetminus\overline{\Lambda^*}})=\mu_{\Lambda,\xi}(\cdot),\quad \text{$\mu$-a.s. $\xi$},
\end{equation}
where $\mathscr{F}_{(\integer^d)^*\smallsetminus\overline{\Lambda^*}}$
is the $\sigma$-algebra generated by $\left\{\eta(b);\,b\in (\integer^d)^*\smallsetminus\overline{\Lambda^*}\right\}$.
Note that the dynamics $\eta_t$ given by \eqref{SDE2} is reversible under
$\mu\in\mathcal{G}$.
We denote the family of $\mu\in\mathcal{G}$ with ergodicity under spatial shifts
by $\mathcal{G}_{\mathrm{ext}}$.

\subsection{Assumptions on Gibbs measures and the surface tension}%
\label{subsec-ass-gibbs}
In order to derive the hydrodynamic limit, we will assume both 
uniqueness of the extremal gradient Gibbs distributions and strict convexity 
of the surface tension.
These assumption are always satisfied under \eqref{cond-convex}, 
cf. see \cite{DGI00} and \cite{FS97}, or for non-convex potential $V$ 
at sufficiently high temperature,
cf. \cite{CD08} and \cite{CDM08}.
On the other hand, at critical temperature,
Biskup and Koteck\'y give an example of gradient Gibbs measures  with two different extremal states, 
cf. \cite{BK07}.
The derivation of the corresponding hydrodynamic 
limit in this case is very challenging open problem.
 

More precisely, let  $\Gamma_N,N\in\natural$ be the periodic lattice $(\integer/N\integer)^d$
and $\Gamma_N^*$ be the set of all directed bonds in $\Gamma_N$.
With $\mathcal{X}_{\Gamma_N}=\{\nabla\phi\in\real^{\Gamma_N^*};\,\phi\in\real^{\Gamma_N}\}$, 
we consider the finite volume Gibbs measure
$\tilde\mu_{N,u}$ on $\mathcal{X}_{\Gamma_N}$ by
\[
	\tilde\mu_{N,u}(d\tilde\eta)=Z_{N,u}^{-1}
	\exp\left(-\frac{1}{2}\sum_{b\in\Gamma_N^*}V(\tilde\eta(b)+u_b)\right)d\tilde\eta,
\]
where $d\tilde\eta$ is Lebesgue measure on $\mathcal{X}_{\Gamma_N}$,
$Z_{N,u}$ is the normalizing constant and
$u_b$ is defined by $u_b=\pm u_i$ for $b=(x\pm e_i,x)$ with $x\in\Gamma_N$
and $1\le i\le d$. We denote the law of $\{\eta(b)+u_b\}$ by $\mu_{N,u}$.

\begin{ass}\label{ass-gibbs-uniq}
%
For each $u\in\real^d$ there exists a unique extremal 
$\mu_u\in\mathcal{G}_{\mathrm{ext}}$ such that
	\[
	E^{\mu_u}[\eta(e_i)]=u_i.
	\]
	Furthermore, it can be obtained as the weak limit of 
	the periodic Gibbs $\mu_{N,u}$ as $N\to\infty$.
\end{ass}

Under Assumption~\ref{ass-gibbs-uniq},  
 the sequence $\{\sigma_N(u)\}$ 
defined by
\[
	\sigma_N(u):=-|\Gamma_N|^{-1}\left(\log Z_{N,u}-\log Z_{N,0}\right),
\]
has a limit.
We thus define the (normalized) surface tension
surface tension $\sigma(u),u\in\real^d$
by 
\begin{equation}\label{def-surface}
\sigma(u)=\lim_{N\to\infty}\sigma_N(u).
\end{equation}
Moreover, we can show the following 
thermodynamic identities between the surface tension and ergodic Gibbs measures:
\begin{gather}
E^{\mu_u}[V'(\eta(e_i))]=\nabla\sigma(u),\quad u\in\real^d,\label{surface-gibbs1} \\
E^{\mu_u}\left[\sum_{i=1}^d \eta(e_i)V'(\eta(e_i))\right]=u\cdot\nabla\sigma(u)+1,\quad u\in\real^d,\label{surface-gibbs2}
\end{gather}
which will be shown
in Section~\ref{subsec-surface-tension}.
%
They play an essential role in the derivation of the hydrodynamic limit.

Further we need some technical assumption on the regularity of $\sigma$ which are well known 
in the strictly convex case \eqref{cond-convex}, cf. \cite{FS97} or 
in the high temperature regime \cite{CD08}.

\begin{ass}\label{ass-surfacetension}
The surface tension $\sigma$ is $C^1$ and $\nabla\sigma:\real^d\to\real^d$ is Lipschitz continuous. Furthermore, $\sigma$ is strictly convex
in the following sense: there exist two constants $C_1,C_2>0$ satisfying
\begin{equation}\label{convex-surface}
C_1|u-v|^2\le (u-v)\cdot(\nabla\sigma(u)-\nabla\sigma(v))
\le C_2|u-v|^2,\quad u,v\in\real^d.
\end{equation}
\end{ass}
\begin{remark}
Note that the convexity of the surface tension, 
alternatively defined in terms of fixed boundary conditions has been established in \cite{KL14}
under very general conditions. 
Moreover, the strict convexity (i.e. lower bound in
\eqref{convex-surface} with $C_1>0$) is not essential
for the hydrodynamic limit since an approximation of $\sigma$ could 
be implemented as in \cite{FS97}. 
\end{remark} 
 
The following example shows that our Assumptions
~\ref{ass-gibbs-uniq} and \ref{ass-surfacetension}
hold in the 
high temperature regime:

\begin{example}
	We introduce a positive parameter $\beta>0$ corresponding to the inverse temperature, 
that is, the  potential $V$ takes the form
	\[V(\eta)=\beta (\tilde{V}_0(\eta)+\tilde{g}(\eta)),\]
	where the symmetric functions $\tilde{V}_0,\tilde g\in C^2(\real)$ satisfy
	$$0<c_{-}\le \tilde V_0''\le c_{+}<\infty\qquad -\infty<-d_{-}<\tilde g''\le d_{+}<\infty$$ for some 
$c_{-}<d_{-}$ and $\|g''\|_{L^q(\real)}<\infty$
	for some $q\ge1$.
	Then for $\beta_0=\beta_0(c_{-},c_{+}+d_{+},\|g''\|_{L^q(\real)})>0$, 
(independent of $d_{-}$!) of the form
$$\beta_0=\frac{(c_{-})^{3q}}{2d\,2^{2q}(c_{+}+d_{+})^{q+1}\|g''\|^{2q}_{L^q(\real)}}$$
	 both Assumptions~\ref{ass-gibbs-uniq} and \ref{ass-surfacetension}
	are satisfied when $\beta\le\beta_0$, see \cite{CD08}
	and its arXiv version (arXiv:0807.2621v1 [math.PR]). 
%
\end{example}

\subsection{Main Result}
The main result in this paper is the following:
\begin{theorem}\label{hydro}
We assume Assumptions~\ref{ass-pot}, \ref{ass-gibbs-uniq} and \ref{ass-surfacetension}.
Furthermore, we assume that there exists 
$h_0\in C^2(D)$ satisfying the following:
\begin{enumerate}
\item 
The function $h_0-f$ has a compact support in $D$.
\item The sequence of initial data $\phi_0=\phi_0^N$
for \eqref{SDE1} satisfies 
\begin{equation}
	\lim_{N\to\infty}E\|h^N(0)-h_0\|_{L^2(D)}^2=0,
\end{equation}
where $h^N(0)$ is the macroscopic height variable corresponding to $\phi_0^N$.
\end{enumerate}
Then, for every $t>0$, $h^N(t)$ converges in $L^2$
as $N\to\infty$ to $h(t)$ which is
the unique weak solution of the partial differential equation (PDE)
\begin{equation}\label{PDE1}
\left\{\begin{split}
	\frac{\partial}{\partial t}h(t,\theta)&=
	\dive\Bigl\{(\nabla\sigma)(\nabla h(t,\theta))
	\Bigr\} \\
	& \equiv\sum_{i=1}^d\frac{\partial}{\partial\theta_i}
	\left\{\frac{\partial\sigma}{\partial u_i}
	\left(\nabla h(t,\theta)\right)\right\}
	,\quad\theta\in D,\,t>0 \\
	h(t,\theta)&=f(\theta),\quad \theta\in D^c,t\ge 0 \\
	h(0,\theta)&=h_0(\theta),\quad \theta\in D,
\end{split}\right.
\end{equation}
where $\nabla h=(\partial h/\partial\theta_i )_{i =1}^d$. Here, the function
$\sigma=\sigma(u)$ is the surface tension. 
More precisely, for every $t>0$,
\begin{equation}\label{eq1.7b}
	\lim_{N\to\infty}E\|h^N(t)-h(t)\|_{L^2(D)}^2=0
\end{equation}
holds.
\end{theorem}
%

\section{Stationary measures and estimate for variance}\label{sec-stationary}
In this section, we mainly discuss properties of 
stationary measures of \eqref{SDE2}
while working on the general assumption,
Assumption~\ref{ass-pot}.
We believe that the results of this section are relevant beyond the derivation of the hydrodynamic limit.
	
\subsection{Generator of \eqref{SDE2} and stationary measures}
We at first note that the infinitesimal generator of \eqref{SDE2} is given by
\begin{equation}\label{generator-infinite}
	\mathscr{L}^{\integer^d}=\sum_{x\in\integer^d}\mathscr{L}_x,
\end{equation}
where
\[\mathscr{L}_x=\sum_{b,b'\in(\integer^d)^*:x_b=x_{b'}=x}\left\{4\frac{\partial^2}{\partial\eta(b)\partial\eta(b')}
-2V'(\eta(b))\frac{\partial}{\partial\eta(b')}\right\}.\]
To keep notation simple, we sometimes denote $\mathscr{L}^{\integer^d}$ by $\mathscr{L}$ if it doesn't cause any confusion.

We can see that the Gibbs property
implies reversibility under \eqref{SDE2}, and
therefore stationarity,
see Proposition~3.1 in \cite{FS97} for details.
We note that the same argument as in \cite{FS97} is applicable in quite general setting, including ours.
In Theorem~2.1 of \cite{FS97}, the equivalence of
the Gibbs property and stationarity is shown using \eqref{cond-convex}, here we show this result using another approach.

\begin{theorem}\label{th2.1}
	We assume Assumption~\ref{ass-pot}.
%
	If $\mu\in\mathcal{P}_2(\mathcal{X})$ is 
	invariant under spatial shift and
	a stationary measure corresponding
	to $\mathscr{L}$,
	i.e.,
	\[\int_{\mathcal{X}}\mathscr{L}f(\eta) \mu(d\eta)=0,
	\quad f\in C_{\mathrm{loc}}^{2}(\mathcal{X}),\]
	then $\mu$ is a Gibbs measure, i.e.,
	\eqref{DLR-eq} 
	holds. 
\end{theorem}

Since the proof of Theorem~\ref{th2.1} is slightly long,
we postpone the proof until the end of this paper,
see Section~\ref{sec-proof-3.1}.

\subsection{Uniform bound for the variance for stationary measures}
If the potential $V$ is a strictly convex function
satisfying \eqref{cond-convex}, we then get the uniform
bound for the variance for Gibbs measures as a direct consequence
of the Brascamp-Lieb inequality. 
See \cite{DGI00} for details.
Our next result based on dynamical approach shows that the variance remains bounded in the tilt 
$u$ for general potentials under Assumption~\ref{ass-pot}.

\begin{theorem}\label{unif-var}
	We assume Assumption~\ref{ass-pot}.
	Let $\mathcal{S}_{\mathrm{ext}}$ be the family of stationary measures for the
	gradient field \eqref{SDE2} which are tempered, translation invariant and ergodic under
	spatial shift.
	The variance of $\eta(b),b\in(\integer^d)^*$
	under $\mu$ are bounded
	from above by a constant independent of $\mu$,
	that is,
	\[\sup_{\mu\in\mathcal{S}_{\mathrm{ext}}}\mathrm{Var}_\mu[\eta(b)]<\infty,\quad
	b\in(\integer^d)^*\]
	holds.
\end{theorem}
\begin{proof}
	We shall show the desired bound by
	arranging the argument of the proof of Proposition~2.1 of \cite{FS97}.
	We fix $\mu\in\mathcal{S}_{\mathrm{ext}}$ and we define the vector
	$u=(u_i)_{1\le i\le d}\in\real^d$ by
	\[u_i=E^{\mu}[\eta((e_i,0))],\quad 1\le i\le d.\]
	Let $\eta_t\in \mathcal{X}$ be the solution of SDEs \eqref{SDE2}
	with initial distribution $\mu$.
	Introducing $\phi_t\in\real^{\integer^d}$ by
	\[\phi_t(0)=\int_{0}^t U_0(\eta_s)\,ds+\sqrt{2}w_t(0)\] 
	and
	\[\phi_t(x)=\phi_t(0)+\sum_{b\in\mathcal{C}_{0,x}}\eta_t(b),\quad x\in\integer^d,\]
	where $C_{0,x}$ is an arbitrary chain connecting $0$ to $x$,
	we then obtain that $\phi_t$ solves the SDEs
	\[d\phi_t(x)=-U_x(\phi_t)\,dt+\sqrt{2}dw_t(x),\quad x\in\integer^d.\]
	Our calculation will be based on the energy estimate for $\phi_t$ introduced above.
	
	Let $\ell\ge1$ and $\Lambda\equiv \Lambda_\ell=[-\ell, \ell]^d\cap\integer^d$.
	For a deterministic $\psi\in\real^{\integer^d}$ with
	\[\psi(x)=u\cdot x,\quad x\in\integer^d,\]
	we obtain
	\begin{align*}
	d\sum_{x\in\Lambda}(\phi_t(x)-\psi(x))^2
	&=-2\sum_{x\in\Lambda}(\phi_t(x)-\psi(x))U_x(\phi_t)\,dt
	+2|\Lambda|\,dt+M_t
	\end{align*}
	with a martingale $M_t$ by It\^o's formula.
	Performing summation-by-parts, we get
	\begin{align*}
		\sum_{x\in\Lambda}(\phi_t(x)-\psi(x))U_x(\phi_t)
		&=\frac{1}{2}\sum_{b\in\overline{\Lambda^*}}(\nabla\phi_t(b)-\nabla\psi(b))V'(\nabla\phi_t(b)) \\
		&\qquad {}-\sum_{b\in\overline{\Lambda^*};\,x_b\in\Lambda^\complement}
		(\phi_t(x_b)-\psi(x_b))V'(\nabla\phi_t(b)).
	\end{align*}
	We thus have
	\begin{equation}\label{eq5.2}
		\sum_{x\in\Lambda}(\phi_T(x)-\psi(x))^2=I_0
		+I_1(T)+I_2(T)+2|\Lambda|T+M_T,
	\end{equation}
	where $I_0, I_1(T)$ and $I_2(T)$ are defined by
	\begin{align*}
	I_0&=\sum_{x\in\Lambda}(\phi_0(x)-\psi(x))^2, \\
	I_1(T)&=-\int_0^T\sum_{b\in\overline{\Lambda^*}}(\nabla\phi_t(b)-\nabla\psi(b))V'(\nabla\phi_t(b))\,dt, \\
	I_2(T)&=2\int_0^T\sum_{b\in\overline{\Lambda^*};\,x_b\in\Lambda^\complement}
		(\phi_t(x_b)-\psi(x_b))V'(\nabla\phi_t(b))\,dt.
	\end{align*}

	From now on, we shall give bounds for expectations of $I_0,I_1(T)$ and $I_2(T)$ separately.

		We at first give a estimate for the expectation of $I_0$. Here, 
		the same argument as the proof
		of (2.14) in \cite{FS97} can be applied. That is,
		from ergodicity and temperedness of $\mu$, we have
		\begin{equation}\label{eq-ergodicity}
			\lim_{|x|\to\infty}\frac{1}{|x|^2}E[(\phi_0(x)-\psi(x))^2]=0,
		\end{equation}
		and this implies that
		\[\lim_{\ell\to\infty}\ell^{-2}|\Lambda|^{-1}E[I_0]=0.\]
		We therefore obtain that for every $\epsilon>0$ there exists $\ell_0\ge1$
		such that
		\begin{equation}\label{bound-I0}
		E[I_0]\le \epsilon\ell^{2}|\Lambda|
		\end{equation}
		holds for every $\ell\ge \ell_0$.

	We shall next calculate $I_1(T)$ and its expectation.
	From Assumption~\ref{ass-pot}, $I_1(T)$ can be calculated as follows:
	{\allowdisplaybreaks\begin{align*}
	I_1(T)
	&= -\int_0^T\sum_{b\in\overline{\Lambda^*}}(\nabla\phi_t(b)-\nabla\psi(b))
	(V_0'(\nabla\phi_t(b))-V_0'(\nabla\psi(b)))\,dt \\
	&\qquad {}-\int_0^T\sum_{b\in\overline{\Lambda^*}}(\nabla\phi_t(b)-\nabla\psi(b))g'(\nabla\phi_t(b))\,dt \\
	&\qquad {}+\int_0^T\sum_{b\in\overline{\Lambda^*}}(\nabla\phi_t(b)-\nabla\psi(b))
	(V_0'(\nabla\psi(b)))\,dt \\
	&\le -c_{-}\int_0^T\sum_{b\in\overline{\Lambda^*}}(\nabla\phi_t(b)-\nabla\psi(b))^2
	\,dt \\
	&\qquad {}+\|g'\|_{\infty}\int_0^T\sum_{b\in\overline{\Lambda^*}}
	\left|\nabla\phi_t(b)-\nabla\psi(b)\right|dt \\
	&\qquad {}+\int_0^T\sum_{b\in\overline{\Lambda^*}}(\nabla\phi_t(b)-\nabla\psi(b))
	V_0'(\nabla\psi(b))\,dt \\
	&=:I_{1,1}(T)+I_{1,2}(T)+I_{1,3}(T).
	\end{align*}}
	Using Schwarz's inequality, we obtain the following estimate for the second term $I_{1,2}(T)$:
	\[I_{1,2}(T)\le 
	\frac{1}{2}\lambda\|g'\|_{\infty}\int_0^T
	\sum_{b\in\overline{\Lambda^*}}\left|\nabla\phi_t(b)-\nabla\psi(b)\right|^2\,dt
	+\frac{1}{2}\lambda^{-1}\|g'\|_{\infty}\left|\overline{\Lambda^*}\right|T
	\]
	for arbitrary $\lambda>0$.  If $\|g'\|_{\infty}>0$ holds, we then have
	\begin{equation}\label{eq5.3}
	I_{1,2}(T)\le 
	\frac{1}{2}c_{-}\int_0^T
	\sum_{b\in\overline{\Lambda^*}}\left|\nabla\phi_t(b)-\nabla\psi(b)\right|^2\,dt
	+\frac{1}{2}c_{-}^{-1}\|g'\|_{\infty}^2\left|\overline{\Lambda^*}\right|T
	\end{equation}
	by taking $\lambda=c_{-}\|g'\|_{\infty}^{-1}$. Note that
	the estimate \eqref{eq5.3} trivially holds when $\|g'\|_{\infty}=0$.
	Summarizing above and taking expectation, we obtain
	\begin{align*}
		E[I_1(T)]\le 
		-\frac{1}{2}c_{-}\int_0^T\sum_{b\in\overline{\Lambda^*}}E[(\nabla\phi_t(b)-\nabla\psi(b))^2]\,dt+\frac{1}{2}c_{-}^{-1}\|g'\|_{\infty}^2\left|\overline{\Lambda^*}\right|T.			
	\end{align*}
	Here, we have used
	\[E[I_{1,3}(T)]=0,\]
	which follows from the definition of $\psi$ and $u$. 
	From the relationship $\nabla\phi_t=\eta_t$, the stationarity of $\mu$ and
	the definition of $u$, we have
	\[E[(\nabla\phi_t(b)-\nabla\psi(b))^2]=\mathrm{Var}_\mu[\eta(b)].\]
	Since $\mu$ is translation invariant, we also have
	\[\sum_{b\in\overline{\Lambda^*}}
		\mathrm{Var}_\mu[\eta(b)]\ge \kappa|\Lambda|\sum_{b:x_b=0}
		\mathrm{Var}_\mu[\eta(b)]\]
	with a constant $\kappa>0$. Applying above, we finally conclude
	\begin{equation}\label{bound-I1}
	E[I_1(T)]\le 
		-\frac{1}{2}c_{-}\kappa T|\Lambda|\sum_{b:x_b=0}
		\mathrm{Var}_\mu[\eta(b)]
		+\frac{1}{2}c_{-}^{-1}\|g'\|_{\infty}^2\left|\overline{\Lambda^*}\right|T.
	\end{equation}

	We next calculate the expected value of $I_2(T)$.
	Putting $\tilde{I}_2(T)$ by
	\[\tilde{I}_2(T)=2\int_0^T\sum_{b\in\overline{\Lambda^*};\,x_b\in\Lambda^\complement}
		(\phi_t(x_b)-\psi(x_b))(V'(\nabla\phi_t(b))-V'(
		\nabla\psi(b)))\,dt,\]
	we have
	\[E[I_2(T)]=E[\tilde{I}_2(T)]\]
	from the definition of $u$. We shall thus calculate $\tilde{I}_2(T)$
	instead of $I_2(T)$. Using Schwarz's inequality, we obtain
	\begin{align}
		E[\tilde{I}_2(T)]&\le
		\gamma\ell^{-1}|\partial \Lambda_\ell^*|\int_0^T
		\sup_{y\in\partial\Lambda}E[(\phi_t(y)-\psi(y))^2]\,dt \nonumber\\
		&\qquad {}+\gamma^{-1} \ell\int_0^T\sum_{b\in\overline{\Lambda^*};\,x_b\in\Lambda^\complement}
		E[(V'(\nabla\phi_t(b))-V'(
		\nabla\psi(b)))^2]\,dt \nonumber
		\\
		&=:F_{2,1}(T)+F_{2,2}(T) \label{eq4.6f}
	\end{align}
	for an arbitrary $\gamma>0$, where $\partial\overline{\Lambda^*}\subset(\integer^d)^*$ and $\partial \Lambda$
	are define by
	\begin{gather*}
		\partial\overline{\Lambda^*}=\left\{b\in\overline{\Lambda^*};\,x_b\in\Lambda^\complement\right\}, \\
		\partial\Lambda=\left\{x_b;\, b\in \partial\overline{\Lambda^*}\right\}.
	\end{gather*}
	For $F_{2,2}(T)$, since $V'$ is Lipschitz continuous, there exists a constant
	$C>0$ such that
	\begin{align}
		F_{2,2}&\le C\gamma^{-1}\ell^dT
		\sum_{b:x_b=0}\mathrm{Var}_\mu[\eta(b)] \label{eq4.7f}
	\end{align}
	by using the translation invariance of $\mu$.
	For $F_{2,1}(T)$, let us use a
	similar argument to the proof of (2.12) in \cite{FS97}.
	Taking $\Lambda'=\Lambda_{[\ell/2]}$, we have
	\begin{align*}
		\left(\phi_t(y)-\psi(y)\right)^2
		&\le 2\left(\phi_t(y)-\psi(y)
		-\frac{1}{|\Lambda'|}\sum_{x\in\Lambda'}
		(\phi_t(x)-\psi(x))\right)^2\\
		&\qquad {}+2\left(
		\frac{1}{|\Lambda'|}\sum_{x\in\Lambda'}(\phi_t(x)-\psi(x))\right)^2\\
		&=:A_1+A_2
	\end{align*}
	for every $y\in\partial\Lambda_{\ell}$.
	For the term $A_1$, the calculations runs quite parallel to the argument in \cite{FS97} and we can obtain that
	for every $\epsilon>0$ there exists $\ell_1\ge1$ such that
	\[E[A_1]\le \epsilon\ell^2\]
	holds for every $\ell\ge \ell_1$. 
	Let us give a bound for the term $A_2$. Using It\^o's formula,
	we obtain
	\begin{align*}
	\frac{1}{|\Lambda'|}\sum_{x\in \Lambda'}(\phi_t(x)-\psi(x))
	&=\frac{1}{|\Lambda'|}\sum_{x\in \Lambda'}(\phi_0(x)-\psi(x))\\
	&\qquad {}-\frac{1}{|\Lambda'|}\int_0^t\sum_{x\in \Lambda'}\sum_{b\in(\integer^d)^*; x_b=x}V'(\eta_s(b))ds+
	\frac{1}{|\Lambda'|}\sum_{x\in \Lambda'}w_t(x) 
	\end{align*}	
	and therefore we get
	\begin{align*}
	A_2
	&\le 4\left(\frac{1}{|\Lambda'|}\sum_{x\in \Lambda'}(\phi_0(x)-\psi(x))
	\right)^2\\
	&\qquad {}+4\left(\frac{1}{|\Lambda'|}\int_0^t\sum_{x\in \Lambda'}\sum_{b\in(\integer^d)^*; x_b=x}V'(\eta_s(b))ds\right)^2+
	4\left(\frac{1}{|\Lambda'|}\sum_{x\in \Lambda'}w_t(x)\right)^2 \\
	&=:A_{2,1}+A_{2,2}+A_{2,3}.
	\end{align*}		
	Similarly to \eqref{bound-I0}, we obtain that
	for every $\epsilon>0$ there exists $\ell_2\ge1$ such that
	\[E[A_{2,1}]\le \epsilon\ell^2\]
	holds for every $\ell\ge \ell_2$. We also obtain
	\[E\left[A_{2,3}\right]=\frac{2t}{|\Lambda'|}
	\]
	by a simple calculation. We shall estimate the term $A_{2,2}$.
	We note that
	\begin{align*}
		\frac{1}{|\Lambda'|}\int_0^t\sum_{x\in \Lambda'}\sum_{b\in(\integer^d)^*; x_b=x}(V'(\eta_s(b))-V'(\psi(b)))ds 
		&=\frac{1}{|\Lambda'|}\int_0^t
		\sum_{b\in B_{\ell}}(V'(\eta_s(b))-V'(\psi(b)))ds,
	\end{align*}
	where 
	\[B_{\ell}=\{b\in(\integer^d)^*; x_b\in\Lambda',y_b\not\in\Lambda'\}.\]
	Here, we have used
	\[\sum_{x\in \Lambda'}\sum_{b\in(\integer^d)^*; x_b=x}V'(\psi(b))=0,\]
	which follows from the definition of $\psi$ and the symmetry of $V$.
	We therefore obtain
	\begin{align*}
	E[A_{2,2}]
	&\le \frac{(c_{+}+C_g)^2|B_{\ell}|^2t^2}{|\Lambda'|^2}\sum_{b:x_b=0}
	\mathrm{Var}_\mu[\eta(b)].
	\end{align*}
	Summarizing above, we conclude the following:
	for every $\epsilon>0$ there exists $L\ge 1$
	such that
	\begin{align}\label{eq3.4}
	\sup_{y\in\partial\Lambda_{\ell}}E[(\phi_t(y)-\psi(y))^2]
	&\le C'\left(\epsilon\ell^2+\ell^{-2}t^2
	\sum_{b:x_b=0}\mathrm{Var}_\mu[\eta(b)]+\ell^{-d}t\right)
	\end{align}
	for every $t\ge0$ and $\ell\ge L$ with a constant $C'>0$.
	Note that the constant $C'$ does not depend on $\mu$ while
	$L$ may depend on $\mu$.
	Combining \eqref{eq4.6f} with \eqref{eq4.7f} and \eqref{eq3.4}, we
	get the following bound for $E[I_2]$: 
	\begin{align}		
		E[I_2(T)]&\le 
		C'\gamma\epsilon^2\ell T\left|\partial \overline{\Lambda^*}\right|+C'\gamma\ell^{-3}T^3\left|\partial\overline{\Lambda^*}\right|
	\sum_{b:x_b=0}\mathrm{Var}_\mu[\eta(b)] \nonumber\\
		&\qquad {}+C\gamma^{-1}T\ell\left|\partial\overline{\Lambda^*}\right|
		\sum_{b:x_b=0}\mathrm{Var}_\mu[\eta(b)]+C'\gamma\left|\partial\overline{\Lambda^*}\right|\ell^{-d-1}T^2 \label{bound-I2}
	\end{align}
	for every $\epsilon>0$ and $\ell$ large enough.

	Inserting \eqref{bound-I0}, \eqref{bound-I1} and \eqref{bound-I2}
	into 
	the expectation of \eqref{eq5.2} divided by $|\Lambda|T$, we obtain
	\begin{align*}
	&\left(\frac{1}{2}c_{-}\kappa-C'\gamma \left|\partial \overline{\Lambda^*}\right| |\Lambda|^{-1}\ell^{-3}T^2-C\gamma^{-1}\ell
	|\Lambda|^{-1}|\partial \overline{\Lambda^*}|\right)
	\sum_{b:x_b=0}
\mathrm{Var}_\mu[\eta(b)] \nonumber \\			
&\quad \le \epsilon^2\ell^2T^{-1}+\frac{1}{2}c_{-}^{-1}\|g'\|_{\infty}^2\left|\overline{\Lambda^*}\right||\Lambda|^{-1}
+C'\gamma\epsilon^2\ell\left|\partial \overline{\Lambda^*}\right| |\Lambda|^{-1} \nonumber \\
&\qquad\qquad {}+C'\gamma\ell^{-d-1}T\left|\partial\overline{\Lambda^*}\right||\Lambda|^{-1}
	\end{align*}
	for every $\epsilon>0$ and $\ell$ large enough.
	Here, taking $T=\gamma^{-1}\ell^2$ and recalling
	the definition of $\Lambda,\partial\overline{\Lambda^*}$ and $\overline{\Lambda^*}$,
	we obtain 
	\begin{align}
	&\left(\frac{1}{2}c_{-}\kappa-C_1\gamma^{-1}\right)
	\sum_{b:x_b=0}\mathrm{Var}_\mu[\eta(b)]
	\le C_2\epsilon^2\gamma+C_3 \label{eq3.8}
	\end{align}
	with constants $C_1,C_2,C_3\ge0$. We emphasize that
	constants appearing on \eqref{eq3.8} does not depend on $\mu$.
	Choosing $\gamma$ large enough such that
	\[\frac{1}{2}c_{-}\kappa-C_1\gamma^{-1}>0,\]
	we conclude the desired bound.
\end{proof}
\begin{remark}
The argument in the proof of Theorem~\ref{unif-var} can be applied
	also to the finite volume Gibbs measures defined in
	Section~\ref{subsec-ass-gibbs}.
	Under Assumption~\ref{ass-pot},
	the variance of $\eta(b),b\in(\integer^d)^*$
	under $\mu_{N,u}$ are bounded
	from above by a constant independent of $u$ and $N$,
	that is,
	\[\sup_{N\ge1}\sup_{u\in\real^d}\mathrm{Var}_{\mu_{N,u}}[\eta(b)]<\infty,\quad
	b\in(\integer^d)^*\]
	holds.
	The above implies that
	the sequence $\{\mu_{N,u}; N\ge1\}$ is tight  for
	given $u\in\real^d$ 
	and every limit point is a
	tempered, translation invariant Gibbs measure.
\end{remark}

\section{The proof of Theorem~\ref{hydro}}
\label{sec-mainthm}
In this section, we shall complete our main 
result, Theorem~\ref{hydro}.
To do so, we at first summarize properties of the surface tension $\sigma$.
The estimate established in the previous sections will play a key role in the proofs. 
After that, we finish the proof of Theorem~\ref{hydro}.

\subsection{A priori bounds for the macroscopic height variable}
We shall derive the bound corresponding to Proposition~4.1 in \cite{N03}.
Once we have Proposition~\ref{prop1.2} and Theorem~\ref{th2.1} in Section~\ref{sec-stationary},
we can follow the argument of \cite{N03} 
assuming that the limit 
of the initial datum is smooth enough.

\begin{prop}\label{prop1.2}
	There exists a constant $K>0$ depending $f$ and $V$ such that
	\[E\left\|h^N(t)\right\|^2_{L^2(D)}
	+c_{-}N^{-d}E\int_0^t\sum_{b\in\overline{D_N^*}}\left(\nabla\phi_s^N(b)\right)^2\,ds
	\le 2E\left\|h^N(0)\right\|^2_{L^2(D)}+K(1+t),\]
	where $\phi_s^N\in\real^{\integer^d}$ is defined by $\phi_s^N(x):=\phi_{N^2s}(x)$ for 
	$x\in\integer^d$.
\end{prop}
\begin{proof}
{\allowdisplaybreaks
Using It\^o's formula, we have
\begin{align*}
	\|h^N(t)-f^N\|_{L^2(D)}^2 
	&=N^{-d-2}\sum_{x\in D_N}\left(\phi^N_0(x)-\psi^N(x)\right)^2 \\
	&\quad\quad {}-2N^{-d}\int_0^t\sum_{x\in D_N}\left(\phi^N_s(x)-\psi^N(x)\right)\partial_xH(\phi_s^N)\,ds \\
	&\quad\quad {}+2N^{-d}|D_N|t+M_t^N,
\end{align*}
where $M_t^N$ is a martingale. Performing the summation-by-parts at the second term in the right hand side,
we obtain
\begin{align*}
	-2N^{-d}&\int_0^t\sum_{x\in D_N}\left(\phi^N_s(x)-\psi^N(x)\right)\partial_xH(\phi_s^N)\,ds \\
	&= -N^{-d}\int_0^t\sum_{b\in \overline{D_N^*}}\nabla\phi^N_s(b)V_0'(\nabla\phi_s^N(b))\,ds \\
	&\qquad {}+N^{-d}\int_0^t\sum_{b\in \overline{D_N^*}}\nabla\psi^N(b)V_0'(\nabla\phi_s^N(b))\,ds \\
	&\qquad {}-N^{-d}\int_0^t\sum_{b\in \overline{D_N^*}}\left(\nabla\phi^N_s(b)-\nabla\psi^N(b)\right)g'(\nabla\phi_s^N(b))\,ds \\
	&=:N^{-d}\int_0^t I_1(s)\,ds+N^{-d}\int_0^tI_2(s)\,ds+N^{-d}\int_0^t I_3(s),ds.
\end{align*}
Here, we have used the boundary condition $\phi_t^N(x)=\psi^N(x)$ for $x\in\integer^d\smallsetminus D_N$ and $t\ge0$.
For the main part $I_1(s)$, we have
\begin{align*}
	I_1(s)\le -c_{-}\sum_{b\in \overline{D_N^*}}\left(\nabla\phi_s^N(b)\right)^2
\end{align*}
from the strict convexity of $V_0$.
Next, we have for $I_2(s)$
\begin{align*}
	\left|I_2(s)\right|
	&\le C_f\sum_{b\in \overline{D_N^*}}
	\left|V_0'(\nabla\phi^N_s(b))\right| \\
	&\le C_fc_{+}\sum_{b\in \overline{D_N^*}}
	\left|\nabla\phi^N_s(b)\right| \\
	&\le \frac{1}{4}c_{-}\sum_{b\in \overline{D_N^*}}
	\left(\nabla\phi^N_s(b)\right)^2
	+4C_f^2c_{+}^2c_{-}^{-1}\Bigl|\overline{D_N^*}\Bigr|,
\end{align*}
where the constant $C_f$ is defined by
\[C_f:=\sup_{1\le i,j\le d}\sup_{\theta\in\real^d}
\left|\frac{\partial f}{\partial \theta_i}(\theta)\right|.\]
Finally, for $I_3(s)$, we have
\begin{align*}
	\left|I_3(s)\right| 
	&\le C_g\sum_{b\in \overline{D_N^*}}\left|\nabla\phi^N_s(b)-\nabla\psi^N(b)\right| \\
	&\le 2\gamma C_g\sum_{b\in \overline{D_N^*}}\left|\nabla\phi^N_s(b)-\nabla\psi^N(b)\right|^2+
	2\gamma^{-1}C_g\Bigl|\overline{D_N^*}\Bigr|\\
	&\le 4\gamma C_g\sum_{b\in \overline{D_N^*}}\left|\nabla\phi^N_s(b)\right|^2
	+4\gamma C_g\sum_{b\in \overline{D_N^*}}\left|\nabla\psi^N(b)\right|^2+
	2\gamma^{-1}C_g\Bigl|\overline{D_N^*}\Bigr|
\end{align*}
for an arbitrary $\gamma>0$. Choosing $\gamma=c_{-}(16C_g)^{-1}$, we have
\begin{align*}
	&\left|\sum_{b\in \overline{D_N^*}}\left(\nabla\phi^N_s(b)-\nabla\psi^N(b)\right)g'(\nabla\phi_s^N(b))\right| \\
	&\qquad \le \frac{1}{4}c_{-}\sum_{b\in \overline{D_N^*}}\left|\nabla\phi^N_s(b)\right|^2
	+\left(\frac{1}{4}c_{-}C_{f}^2+32C_g^2c_{-}^{-1}\right)\Bigl|\overline{D_N^*}\Bigr|.
\end{align*}
Summarizing above, we get
\begin{align*}
\|h^N(t)-f^N\|_{L^2(D)}^2
	&\le \|h^N(0)-f^N\|_{L^2(D)}^2
	-\frac{1}{2}c_{-}N^{-d}\int_0^t\sum_{b\in \overline{D_N^*}}\left(\nabla\phi_s^N(b)\right)^2\,ds \\
	&\quad\quad {}+\left(4C_f^2c_{+}^2c_{-}^{-1}+\frac{1}{4}c_{-}C_{f}^2+32C_g^2c_{-}^{-1}\right)N^{-d}\Bigl|\overline{D_N^*}\Bigr|t\\
	&\quad\quad {}+2N^{-d}|D_N|t+M_t^N. 
\end{align*}
Taking the expectation, we obtain the conclusion.}
\end{proof}

\subsection{Surface tension and thermodynamic identities}
\label{subsec-surface-tension}
In this subsection, we verify 
several properties of surface tension $\sigma$. 
Note that in view of our estimate of the variance,
Theorem~\ref{unif-var}, we easily can get the identity \eqref{surface-gibbs1}
following the argument of \cite{FS97}. 
The proof of the second equality \eqref{surface-gibbs2} is more delicate, since, in view of the missing higher moment estimate,  we cannot apply immediately apply the argument of \cite{FS97}.

\begin{prop}\label{prop-surface-gibbs2}
	For every translation-invariant, ergodic Gibbs measure $\mu_u$, we have
	\[E^{\mu_u}\left[\sum_{i=1}^d \eta(e_i)V'(\eta(e_i))\right]=u\cdot \nabla\sigma(u)+1.\]
\end{prop}
\begin{proof}
Let $\ell\ge1$ and we denote $\Lambda_\ell$ simply by $\Lambda$.
We define $\psi\in\real^{\integer^d}$ by
$\psi(x)=u\cdot x$ for $x\in\integer^d$.
We at first note that
\begin{align*}
	\sum_{b\in\overline{\Lambda^*}}
	(\eta(b)-\nabla \psi(b))V'(\eta(b))
	&=2\sum_{x\in\Lambda}(\phi^{0,\eta}(x)-\psi(x))
	\frac{\partial H}{\partial \phi(x)}(\phi^{0,\eta})
	\\
	&\quad {}
	-2
	\sum_{b\in\overline{\Lambda^*},x_b\not\in\Lambda}
	(\phi^{0,\eta}(x_b)-\psi(x_b))V'(\eta(b))
\end{align*}
holds by the summation-by-parts.
Since we have
\begin{align*}
	E^{\mu_u}&\left[2\sum_{x\in\Lambda}(\phi^{0,\eta}(x)-\psi(x))
	\frac{\partial H}{\partial \phi(x)}(\phi^{0,\eta})\right] 
	\\
	&
	=E^{\mu_u}\left[E^{\mu_{\Lambda,\xi}}\left[2\sum_{x\in\Lambda}(\phi^{0,\eta}(x)-\psi(x))
	\frac{\partial H}{\partial \phi(x)}(\phi^{0,\eta})\right]\right]
	\\
	&=2|\Lambda|-2
\end{align*}
from the DLR equation and the integration-by-parts,
we obtain
\begin{align*}
E^{\mu_u}&\left[\sum_{b\in\overline{\Lambda^*}}(\eta(b)-\nabla\psi(b))V'(\eta(b))\right]
\\
&
= 	2|\Lambda|-2-2E^{\mu_u}\left[
	\sum_{b\in\overline{\Lambda^*},x_b\not\in\Lambda}
	(\phi^{0,\eta}(x_b)-\psi(x_b))V'(\eta(b))\right].
\end{align*}
On the other hand, we have
\begin{align*}
	E^{\mu_u}&\left[\sum_{b\in\overline{\Lambda^*}}(\eta(b)-\nabla\psi(b))V'(\eta(b))\right] \\
	&=E^{\mu_u}\left[\sum_{b\in\overline{\Lambda^*}}\eta(b)V'(\eta(b))\right]-\sum_{b\in\overline{\Lambda^*}}\nabla\psi(b)
	E^{\mu_u}\left[V'(\eta(b))\right],
\end{align*}
and therefore we obtain 
\begin{align*}
	E^{\mu_u}\left[\sum_{b\in\overline{\Lambda^*}}\eta(b)V'(\eta(b))\right]&=
	\sum_{b\in\overline{\Lambda^*}}\nabla\psi(b)
	E^{\mu_u}\left[V'(\eta(b))\right]+2|\Lambda|-2 \\
	&\quad {}-2E^{\mu_u}\left[
	\sum_{b\in\overline{\Lambda^*},x_b\not\in\Lambda}
	(\phi^{0,\eta}(x_b)-\psi(x_b))V'(\eta(b))\right].
\end{align*}
Since we have
\[
\lim_{\ell\to\infty}|\Lambda|^{-1}E^{\mu_u}\left[\sum_{b\in\overline{\Lambda^*}}\eta(b)V'(\eta(b))\right]=
2E^{\mu_u}\left[\sum_{i=1}^d\eta(e_i)V'(e_i)\right]
\]
by translation-invariance of $\mu_u$ and also have
\[
\lim_{\ell\to\infty}|\Lambda|^{-1}\sum_{b\in\overline{\Lambda^*}}\nabla\psi(b)
	E^{\mu_u}\left[V'(\eta(b))\right]=
	2u\cdot \nabla\sigma(u)
\]
by translation-invariance of $\mu_u$ and the identity \eqref{surface-gibbs1},
we obtain the conclusion once we have
\begin{equation}\label{eq1}
\lim_{\ell\to\infty}\ell^{-d}
E^{\mu_u}\left[
\sum_{b\in\overline{\Lambda^*},x_b\not\in\Lambda}
(\phi^{0,\eta}(x_b)-\psi(x_b))V'(\nabla\phi(b))\right]=0.
\end{equation}
By Schwarz's inequality, we have
\begin{align*}
&\left| \ell^{-d}E^{\mu_u}\left[
\sum_{b\in\overline{\Lambda^*},x_b\not\in\Lambda}
(\phi^{0,\eta}(x_b)-\psi(x_b))V'(\nabla\phi(b))\right]\right| \\
&\quad \le \gamma \ell^{-d}\sup_{x\in\partial\Lambda}E^{\mu_u}[|\phi^{0,\eta}(x)-\psi(x))|^2]+
\gamma^{-1} \ell^{-d}E^{\mu_u}\left[
\sum_{b\in\overline{\Lambda^*},x_b\not\in\Lambda}
|V'(\nabla\phi(b))|^2\right] \\
&\quad \le \gamma \ell^{-d}\sup_{x\in\partial\Lambda}E^{\mu_u}[|\phi^{0,\eta}(x)-\psi(x))|^2]+
\gamma^{-1}\ell^{-d}C_{+}^2|\partial\Lambda| E^{\mu_u}\left[
\sum_{b\in\overline{\Lambda^*},x_b=0}
\nabla\phi(b)^2\right]
\end{align*}
for an arbitrary $\gamma>0$.
Let us estimate the first term in the right hand side.
Let us take $\epsilon>0$ arbitrarily. We can then take $\ell_0\ge1$ such that \eqref{eq3.4}
with $t=0$ holds for every $\ell\ge\ell_0$.
Choosing $\gamma$ as $\gamma=\ell^{-1}\epsilon^{-1}$, we obtain
\begin{align*}
&\left| \ell^{-d}E^{\mu_u}\left[
\sum_{b\in\overline{\Lambda^*},x_b\not\in\Lambda}
(\phi^{0,\eta}(x_b)-\psi(x_b))V'(\nabla\phi(b))\right]\right|\\
&\qquad \le
\epsilon\ell^{-d+1}+\epsilon\ell^{-d+1}C_{+}^2|\partial\Lambda| E^{\mu_u}\left[
\sum_{b\in\overline{\Lambda^*},x_b=0}
\nabla\phi(b)^2\right],
\end{align*}
which shows \eqref{eq1} since $\ell^{-d+1}|\partial\Lambda|$ is bounded
uniformly in $\ell$.
\end{proof}

Finally, we shall establish  
similar decomposition for $\nabla\sigma$ as in Section~3.3 of \cite{N03}.
Applying the arguments there, we can obtain the
uniform $L^p$-bound with $p>2$ and the oscillation inequality
for the discrete version of \eqref{PDE1}.
\begin{prop}
	\label{prop-decomp-surface}
	There exist a $\real^d$-valued
	function $a(u)=(a_i(u))_{1\le i\le d}\in L^\infty(\real^d)^d$
	and a matrix-valued function $A(u)=(A_{ij}(u))_{1\le i,j\le d}$
	satisfying
	\begin{equation}\label{ellipticity}
		c_{-}\mathbb{I}\le A(u)\le c_{+}\mathbb{I},\quad u\in\real^d
	\end{equation}
	such that the identity
	\begin{equation}\label{decomp-surface}
		\nabla\sigma(u)=A(u)u+a(u),\quad u\in\real^d
	\end{equation}
	holds.
\end{prop}
\begin{proof}
	We at first recall the relationship between the surface tension $\sigma$
	and the Gibbs measures:
	\[\nabla_i\sigma(u)=E^{\mu_u}[V'(\eta(e_i))],\quad 1\le i\le d,\]
	where $\mu_u$ is the ergodic Gibbs measure with mean $u\in\real^d$.
	Using $V_0$ and $g$ in Assumption~\ref{ass-pot}, we shall
	take $A(u)$ as
	\[
	A_{ij}(u)=E^{\mu_u}\left[\int_0^1 V_0''(\eta(e_i)-\lambda u_i)\,d\lambda\right]	\delta_{ij},\quad 1\le i,j\le d
	\]
	and $a(u)$ as
	\[
		a_i(u)=E^{\mu_u}\left[V_0'(\eta(e_i)-u_i)\right]
		+E^{\mu_u}[g'(\eta(e_i))],\quad 1\le i\le d.
	\]
	It is easy to verify \eqref{decomp-surface} and
	\eqref{ellipticity} with $A(u)$ and $a(u)$ defined above
	by using Assumption~\ref{ass-pot}. 
	Furthermore, the property $a(u)\in L^\infty(\real^d)^d$
	is an immediate
	consequence of Theorem~\ref{unif-var} and Schwarz's inequality.
\end{proof}

\subsection{Derivation of the macroscopic equation}
%
We shall at first summarize the properties satisfied by
the solution $\bar{h}^N$ for the discretized PDE introduced in Section~3.1
of \cite{N03}.
Since we have assumed the strict convexity of $\sigma$
at Assumption~\ref{ass-surfacetension}, we obtain a priori bounds
for $\bar{h}^N$ in $C([0,T],L^2(D))$ and $L^2([0,T],H^1(D))$,
see Proposition~3.1, Corollary~3.2 in \cite{N03}.
Furthermore,
since we have
\eqref{decomp-surface}, we obtain
the uniform bound of $\nabla\bar{h}^N$ in $L^p([0,T]\times D)$
and the oscillation inequality,
see Propositions~3.3 and 3.4 in \cite{N03}.

We next verify the coupled local equilibrium.
Applying Theorem~\ref{th2.1}, Proposition~\ref{prop1.2}
and the bound for $\bar{h}^N$ stated above, we see
that Proposition~4.2 in \cite{N03} is still valid.
Summarizing above, we obtain that the arguments in Section~4.4 works completely, and we can finally conclude Theorem~\ref{hydro}.

\section{Proof of Theorem~\ref{th2.1}}\label{sec-proof-3.1}
Our proof follows the argument of \cite{F82a}
however special care is required in view of the 
non ellipticity of the  generator $\mathscr{L}$.
\subsection{Generator and Dirichlet form}
In this section, we mainly discuss properties of 
stationary measures of \eqref{SDE2}
	while working on the general assumption, Assumption~\ref{ass-pot}.

Since the calculation is based on the generator and the Dirichlet form,
we shall introduce them before starting discussion. 

We recall that the infinitesimal generator of \eqref{SDE2} is given by
\begin{equation}\label{generator-infinite-a}
	\mathscr{L}^{\integer^d}=\sum_{x\in\integer^d}\mathscr{L}_x,
\end{equation}
where
\[\mathscr{L}_x=\sum_{b,b'\in(\integer^d)^*:x_b=x_{b'}=x}\left\{4\frac{\partial^2}{\partial\eta(b)\partial\eta(b')}
-2V'(\eta(b))\frac{\partial}{\partial\eta(b')}\right\}.\]
We also recall that we sometimes denote $\mathscr{L}^{\integer^d}$ by $\mathscr{L}$ for simplicity.

We also introduce the finite version of \eqref{generator-infinite-a}. We 
at first introduce the state space for that.
We define $\mathcal{X}_{\Lambda}\subset\real^{\Lambda^*}$
for a finite set $\Lambda\subset\integer^d$ by
\[\mathcal{X}_{\Lambda}=\{\eta\equiv\nabla\phi\in\real^{\Lambda^*};\, \phi\in\real^{\Lambda}\}.\]
Note that $\mathcal{X}_{\Lambda,\xi}$ introduced in Section~\ref{sec-notations}
and $\mathcal{X}_{\Lambda}$ are state spaces
for the dynamics with the boundary condition given by $\xi$ and free boundary condition,
respectively.
For a finite set $\Lambda\subset\integer^d$
we define the differential operator $\mathscr{L}^{\Lambda}$ and $\mathscr{L}^{\Lambda,\mathrm{f}}$ by
\begin{equation}\label{generator-bc}
	\mathscr{L}^{\Lambda}=\sum_{x\in\Lambda}\mathscr{L}_x
\end{equation}
and
\begin{equation}\label{generator-free}
	\mathscr{L}^{\Lambda,\mathrm{f}}=\sum_{x\in\Lambda}\mathscr{L}^{\Lambda,\mathrm{f}}_x,
\end{equation}
respectively. Here, $\mathscr{L}^{\Lambda,\mathrm{f}}_x$ is the operator defined by
\[\mathscr{L}^{\Lambda,\mathrm{f}}_x=\sum_{b,b'\in\Lambda^*:x_b=x_{b'}=x}\left\{4\frac{\partial^2}{\partial\eta(b)\partial\eta(b')}
-2V'(\eta(b))\frac{\partial}{\partial\eta(b')}\right\}.\]
The former is the generator
associated to the dynamics $\eta^{\Lambda}_t$ on $\mathcal{X}_{\Lambda,\xi}$ for given
$\xi\in\mathcal{X}$, which is governed by SDEs
\[
\begin{cases}
d\eta^{\Lambda}_t(b)
=-\nabla U_{\cdot}(\eta^{\Lambda}_t)(b)\,dt
+\sqrt{2}dw_t(b),&b\in\overline{\Lambda^*},\\
\eta^{\Lambda}_t(b)\equiv\xi(b),& b\in(\integer^d)^*\smallsetminus\overline{\Lambda^*}.
\end{cases}
\]
The dynamics $\eta^{\Lambda}_t$ is reversible under the finite volume
Gibbs measure $\mu_{\Lambda,\xi}$ introduced in Section~\ref{sec-notations}.
The latter is the generator
associated to $\eta^{\Lambda,\mathrm{f}}_t$ governed by SDEs
\[d\eta^{\Lambda,\mathrm{f}}_t(b)
=-\nabla U_{\cdot}^\Lambda(\eta^{\Lambda,\mathrm{f}}_t)(b)\,dt
+\sqrt{2}dw_t(b),\quad b\in\Lambda^*,\]
on $\mathcal{X}_{\Lambda^*}$,
which corresponds to the dynamics \eqref{SDE2} with free boundary condition.
The character ``$\mathrm{f}$'' in notions means free boundary condition.
Here, $H^{\Lambda}$ and $U^{\Lambda}_x$ are defined by
\begin{gather*}
	H^{\Lambda}(\eta)=\sum_{b\in\Lambda^*}V(\eta(b)),\\
	U^{\Lambda}_x(\eta)=\sum_{b\in\Lambda^*;x_b=x}V'(\eta(b))
\end{gather*}
for $\eta\in\mathcal{X}_{\Lambda^*}$.
Note that $\eta^{\Lambda,\mathrm{f}}_t$ is also reversible and the reversible measure is
\begin{gather*}
\mu_{\Lambda,\mathrm{f}}(d\eta)=Z^{-1}_{\Lambda,\mathrm{f}}\exp\left(-H^{\Lambda}(\eta)\right)d\eta_{\Lambda^*},
\end{gather*}
on $\mathcal{X}_{\Lambda^*}$, where
$d\eta_{\Lambda^*}$ is the Lebesgue measure on $\mathcal{X}_{\Lambda^*}$ and $Z_{\Lambda,\mathrm{f}}$ is the normalizing constant.
The Dirichlet form associated to $\eta^{\Lambda,\mathrm{f}}_t$ is given by
\[\mathscr{E}^{\Lambda,\mathrm{f}}(f,g)=\int\sum_{x\in\Lambda}\left(\sum_{b\in\Lambda^*:x_b=x}\frac{\partial f}{\partial \eta(b)}\right)
\left(\sum_{b\in\Lambda^*:x_b=x}\frac{\partial g}{\partial \eta(b)}\right)\mu_{\Lambda,\mathrm{f}}(d\eta),\]
which plays a key role in the proof of the main theorem in this section, Theorem~\ref{th2.1}.

\subsection{Proof of Theorem~\ref{th2.1}}
In this subsection, let us complete the proof
of Theorem~\ref{th2.1}, which is based on the method of \cite{F82a}. Main tool is 
the integration-by-parts formula for $\mathscr{L}^{\Lambda,\mathrm{f}}$ and entropy production rate.
For the computation, we introduce a small lemma:
\begin{lemma}\label{lem1.1}
		Let $\Lambda$ be a finite subset of $\integer^d$.
	\begin{enumerate}
		\item Let $p:\real\to\real$ be a probability density on $\real$. Then,
		the image measure of $p(\phi(0))\prod_{x\in\Lambda}d\phi(x)$ by
		the discrete gradient $\nabla$ is nothing but 
		the Lebesgue measure on $\mathcal{X}_{\Lambda^*}$.
		\item If $f:\real^{\Lambda}\to\real$ is the form $f(\phi)=F(\nabla\phi)$, then
		\[\frac{\partial f}{\partial \phi(x)}=2\sum_{b:x_b=x}\frac{\partial F}{\partial \eta(b)}(\nabla\phi)\]
		holds. Especially, if $F$ is $\mathscr{F}_{\Lambda^*}$-measurable, we have
		\[\frac{\partial f}{\partial \phi(x)}=2\sum_{b\in\Lambda^*:x_b=x}\frac{\partial F}{\partial \eta(b)}(\nabla\phi)\]		
	\end{enumerate}
\end{lemma}
\begin{proof}
	It is easy to see that
	\[\int F(\nabla\phi)\delta_{u}(d\phi(0))\prod_{x\in\Lambda\smallsetminus\{0\}}d\phi(x)
	=\int F(\nabla\phi)\delta_{v}(d\phi(0))\prod_{x\in\Lambda\smallsetminus\{0\}}d\phi(x)\]
	for every $u,v\in\real$ and bounded $F:\mathcal{X}_{\Lambda^*}\to\real$, which
	indicates that the integral
	\[\int F(\nabla\phi)p(d\phi(0))\prod_{x\in\Lambda}d\phi(x)\]
	does not depend in the choice of a probability density $p$.
	Now, we check that
	the image measure has uniformity in  $\mathcal{X}_{\Lambda}$.
	For $\xi\in\mathcal{X}_{\Lambda^*}$, there exists the $\psi\in\real^{\Lambda}$
	such that $\psi(0)=0$ and 
	\[\xi(b)=\nabla\psi(b)\]
	holds. For a bounded function $F:\mathcal{X}_{\Lambda_n^*}\to\real$, we have
	\begin{align*}
	\int F(\nabla\phi+\xi)p(\phi(0))\prod_{x\in\Lambda}d\phi(x)&=
	\int F(\nabla(\phi+\psi))p(\phi(0))\prod_{x\in\Lambda}d\phi(x) \\
	&=\int F(\nabla\phi)p(\phi(0))\prod_{x\in\Lambda}d\phi(x),
	\end{align*}
	which shows the first assertion.
	
	For $F=F(\nabla\phi)$, we obtain
	\[\frac{\partial F}{\partial\phi(x)}=\sum_{b:x_b=x}\frac{\partial F}{\partial\eta(b)}
	-\sum_{b:y_b=x}\frac{\partial F}{\partial\eta(b)}=2\sum_{b:x_b=x}\frac{\partial F}{\partial\eta(b)},\]
	which shows the second assertion.
\end{proof}

Let us start to prove Theorem~\ref{th2.1}.
We at first introduce $\Phi_{\lambda}:\real\to\real$ by
\[\Phi_{\lambda}(u)=\frac{\lambda}{a}\left(1+(\lambda u)^2\right)^{-m},\]
where
\[a=\int_{\real}(1+u^2)^{-m}du.\]
For $\Lambda_{n}:=[-n,n]^d\cap\integer^d$
we define $\Phi^{\lambda}_{n}:\mathcal{X}_{\Lambda_n^*}\to\real$ by
\[\Phi^{\lambda}_{n}(\eta)=\prod_{x\in\Lambda_n}\Phi_{\lambda}(\phi^{\eta,0}(x)),\]
where $\phi^{\eta,a}$ is the height variable satisfying $\nabla\phi^{\eta,a}=\eta$ and
$\phi^{\eta,a}(0)=a$. Note that $\phi^{\eta,a}$ is uniquely determined by $\eta$ and $a$.
We also define $p^\lambda_n(\eta)$ by
\[p^\lambda_n(\eta)=\int \Phi^{\lambda}_n(\eta-\xi)\mu(d\xi).\]
Applying Lemma~\ref{lem1.1}, we can easily verify that $p^\lambda_n(\eta)$ is probability density
on $\real^{\mathcal{X}_{\Lambda_n^*}}$.
Let $\Psi^\lambda_n(\eta,\xi)=\Phi^{\lambda}_{n}(\xi-\eta)$. Since $\Psi^\lambda_n(\cdot,\xi)\in C^2_{\mathrm{loc}}(\mathcal{X})$, we have
\[\int \mathscr{L}\Psi^\lambda_n(\cdot,\xi)(\eta)\mu(d\eta)=0.\]
Multiplying 
$F(\xi)\in C^2_{\mathrm{loc}}(\mathcal{X})$ whose support is in $\Lambda_n^{*}$, and
integrating in $\xi$ by the uniform measure on $\mathcal{X}_{\Lambda_n^*}$, we obtain
\begin{equation}\label{eq2.1}
	\iint F(\xi)\mathscr{L}\Psi^\lambda_n(\cdot,\xi)(\eta)
	\mu(d\eta)d\xi_{\Lambda_n^*}=0.
\end{equation}
Applying Lemma~\ref{lem1.1}, the right hand side is calculated as follows:
\begin{align*}
	\iint F(\nabla\psi)\sum_{x\in\integer^d}\left(\frac{\partial^2 \Psi^\lambda_n(\eta,\nabla\cdot)}{\partial\psi(x)^2}
	+\left(\sum_{b\in(\integer^d)^*:x_b=x}V'(\eta(b))\right)\frac{\partial \Psi^\lambda_n(\eta,\nabla\cdot)}{\partial\psi(x)}\right)
	\nu_{\Lambda_n,p}(d\psi)\mu(d\eta),
\end{align*}
where $\nu_{\Lambda_n,p}$ is the measure on $\real^{\Lambda_n}$ defined by
\[\nu_{\Lambda_n,p}(d\psi)=p(\psi(0))\prod_{x\in\Lambda_n}d\psi(x)\]
with a probability density $p$ on $\real$.
Here, we have used the relationship
\begin{gather*}
	\frac{\partial\Psi^\lambda_n(\nabla\cdot,\nabla\psi)}{\partial\phi(x)}(\phi)
	=-\frac{\partial \Phi^{\lambda}_n}{\partial\phi(x)}(\nabla\psi-\nabla\phi)=
	-\frac{\partial\Psi^\lambda_n(\nabla\phi,\nabla\cdot)}{\partial\psi(x)}(\psi) \\
	\frac{\partial^2\Psi^\lambda_n(\nabla\cdot,\nabla\psi)}{\partial\phi(x)^2}(\phi)
	=\frac{\partial^2 \Phi^{\lambda}_n}{\partial\phi(x)^2}(\nabla\psi-\nabla\phi)
	=\frac{\partial^2\Psi^\lambda_n(\nabla\phi,\nabla\cdot)}{\partial\psi(x)^2}(\psi)
\end{gather*}
for $x\in\integer^d$ by the symmetry of $\Phi^{\lambda}$. Noting
\[\frac{\partial\Phi^{\lambda}_n}{\partial\phi(x)}\equiv0,\quad x\in\Lambda_n^\complement,\]
we obtain that the right hand side of \eqref{eq2.1} is computed as follows:
\begin{align*}
	&\iint F(\nabla\psi)\sum_{x\in\Lambda_n}\frac{\partial^2 \Psi^\lambda_n(\eta,\nabla\cdot)}{\partial\psi(x)^2}\nu_{\Lambda_n,p}(d\psi)\mu(d\eta) \\
	&\qquad {}+
	\iint F(\nabla\psi)\sum_{x\in\Lambda_n}\left(\sum_{b\in(\integer^d)^*:x_b=x}V'(\eta(b))\right)\frac{\partial \Psi^\lambda_n(\eta,\nabla\cdot)}{\partial\psi(x)}\nu_{\Lambda_n,p}(d\psi)\mu(d\eta) \\
	&=:I_1+I_2.
\end{align*}

We shall first  calculate $I_1$. Performing integration-by-parts in $\psi$, we have
\begin{align}
	I_1&=-\iint \sum_{x\in\Lambda_n} \frac{\partial F(\nabla\cdot)}{\partial \psi(x)} 
	\frac{\partial \Psi^\lambda_n(\eta,\nabla\cdot)}{\partial\psi(x)}\nu_{\Lambda_n,p}(d\psi)\mu(d\eta)  
	\label{eq2.2} \\
	&\qquad {}-\iint F(\nabla\psi) 
	\frac{\partial \Psi^\lambda_n(\eta,\nabla\cdot)}{\partial\psi(0)}p'(\psi(0))
	\prod_{x\in\Lambda_n}d\psi(x)\mu(d\eta) . \nonumber
\end{align}
Noting that integrands of $I_1$ and the first term in the right hand side of \eqref{eq2.2} are
function of $\nabla\psi$, each integral does not depend on the choice of $p$ by Lemma~\ref{lem1.1}
and therefore the second term does not also. On the other hand, 
since the second term converges to zero if taking the limit $p\to0$ with $p'\to0$, we conclude that the second term must be zero.

Let us choose $F$ as
\begin{equation}\label{eq2.1c}
	F(\nabla\psi)=f\left(\frac{p^\lambda_n(\nabla\psi)}{q_n(\nabla\psi)}\right),
\end{equation}
with some bounded smooth function $f:\real\to\real$ and
\[q_n(\eta)=\exp\left(-H^\Lambda(\eta)\right),\quad \eta\in\mathcal{X}_{\Lambda_n^*}.\]
Noting
\[\frac{\partial p^\lambda_n(\nabla\cdot)}{\partial \psi(x)}=\int\frac{\partial \Psi^\lambda_n(\eta,\nabla\cdot)}{\partial\psi(x)}\mu(d\eta),\]
we have
\begin{align}
I_1
&=-\sum_{x\in\Lambda_n}\int  f'\left(\frac{p^\lambda_n(\nabla\psi)}{q_n(\nabla\psi)}\right)
\left(\frac{\partial }{\partial \psi(x)}\left(\frac{p^\lambda_n(\nabla\cdot)}{q_n(\nabla\cdot)}\right)\right)^2
	q_n(\nabla\psi)\nu_{\Lambda_n,p}(d\psi) \nonumber \\
	&\qquad {}+\sum_{x\in\Lambda_n}\int f'\left(\frac{p^\lambda_n(\nabla\psi)}{q_n(\nabla\psi)}\right)U^\Lambda_{x}(\nabla\psi)
	p^\lambda_n(\nabla\psi)
	\nu_{\Lambda_n,p}(d\psi). \label{eq2.2c}
\end{align}

Next, we shall compute $I_2$. Performing the integration-by-parts in $\psi(x)$ again, we have
{\allowdisplaybreaks
\begin{align*}
	I_2
	&=-\sum_{x\in\Lambda_n}\iint \frac{\partial F(\nabla\psi)}{\partial\psi(x)}
	U_{x}(\eta)
	\Psi^\lambda_n(\eta,\nabla\psi)\nu_{\Lambda_n,p}(d\psi)\mu(d\eta) \\
	&=-\sum_{x\in\Lambda_n}\iint \frac{\partial F(\nabla\psi)}{\partial\psi(x)}
	\left(U_x(\eta)-U^\Lambda_x(\eta)\right)
	\Psi^\lambda_n(\eta,\nabla\psi)\nu_{\Lambda_n,p}(d\psi)\mu(d\eta) \\
	& \qquad {}-\sum_{x\in\Lambda_n}\iint \frac{\partial F(\nabla\psi)}{\partial\psi(x)}
	\left(U_x^\Lambda(\eta)-U_x^\Lambda(\nabla\psi)\right)
	\Psi^\lambda_n(\eta,\nabla\psi)\nu_{\Lambda_n,p}(d\psi)\mu(d\eta) \\
	& \qquad {}-\sum_{x\in\Lambda_n}\iint \frac{\partial F(\nabla\psi)}{\partial\psi(x)}
	U_x^\Lambda(\nabla\psi)
	p^\lambda_n(\nabla\psi)\nu_{\Lambda_n,p}(d\psi).
\end{align*}}
Summarizing \eqref{eq2.1}, \eqref{eq2.2c} and above, we obtain
\begin{align}\label{eq2.3c}
	\sum_{x\in\Lambda_n}F^\lambda_x(n,f)
	&=-\sum_{x\in\Lambda_n}\iint \frac{\partial F(\nabla\psi)}{\partial\psi(x)}
	\left(U_x(\eta)-U^\Lambda_x(\eta)\right)
	\Psi^\lambda_n(\eta,\nabla\psi)\nu_{\Lambda_n,p}(d\psi)\mu(d\eta) \\
	&\qquad {}-\sum_{x\in\Lambda_n}\iint \frac{\partial F(\nabla\psi)}{\partial\psi(x)}
	\left(U_x^\Lambda(\eta)-U_x^\Lambda(\nabla\psi)\right)
	\Psi^\lambda_n(\eta,\nabla\psi)\nu_{\Lambda_n,p}(d\psi)\mu(d\eta) \nonumber\\
	&=: \sum_{x\in\Lambda_n}R^\lambda_{1,x}(n,f)+\sum_{x\in\Lambda_n}R^\lambda_{2,x}(n,f) \nonumber
\end{align}
if we take $F$ as in \eqref{eq2.1c}, where $F^\lambda_x(n,f)$ is defined by
\begin{equation}\label{eq2.1b}
F^\lambda_x(n,f):=\int f'\left(\frac{p^\lambda_n(\nabla\phi)}{q_n(\nabla\phi)}\right)
\left(\frac{\partial}{\partial \phi(x)}\left(\frac{p^\lambda_n(\nabla\cdot)}{q_n(\nabla\cdot)}\right)\right)^2q_n(\nabla\phi)\nu_{n,p}(d\phi)
\end{equation}
We note that if we can take $f(u)=\log u$, the left hand side coincides with the entropy production rate,
that is,
\begin{equation}\label{ent-prod} 
\sum_{x\in\Lambda_n}F^\lambda_x(n,f)=\mathscr{E}^{\Lambda,\mathrm{f}}\left(\sqrt{r^\lambda_n},\sqrt{r^\lambda_n}\right)
\end{equation}
holds, where $r_n$ is the probability density with respect to $\mu^{\Lambda,\mathrm{f}}$ given by
\[r^\lambda_n=Z_{\Lambda,\mathrm{f}}p^\lambda_n q_n^{-1}.\]
In this case, we simply denote $F^\lambda_{x}(n,f)$ by $F^\lambda_{x}(n)$.
Before continuing the discussion for $\eqref{eq2.3c}$, 
we shall verify the integrability of integrands in \eqref{eq2.1b} with $f(u)=\log u$.
\begin{lemma}
	For every $n,\lambda,x$, the integral \eqref{eq2.1b} is finite if $f(x)=\log x$.
\end{lemma}
\begin{proof}
	We at first note that
	\[F^\lambda_x(n)=\int\left((p^\lambda_n)^{-1}\frac{\partial p^\lambda_n(\nabla\cdot)}{\partial \phi(x)}
	-q_n^{-1}\frac{\partial q_n(\nabla\cdot)}{\partial \phi(x)}\right)^2p^\lambda_n(\phi)\nu_{n,p}(d\phi).\]
	
	By the definition of $\Phi_{\lambda}$, we have for $x\neq 0$
	\begin{align*}
		\frac{\partial p^\lambda_n(\nabla\cdot)}{\partial \phi(x)} 
		&=\int (-2m\lambda^2)\frac{\phi(x)-\phi(0)-\phi^{\xi,0}(x)}{1+\lambda^2(\phi(x)-\phi(0)-\phi^{\xi,0}(x))^2}
		\Phi^{\lambda}_n(\xi-\nabla\phi)\,\mu(d\xi)
	\end{align*}
	and
	\begin{align*}
	\left|\frac{\partial p^\lambda_n(\nabla\cdot)}{\partial \phi(x)}\right|
	\le & m\lambda\int \Phi^{\lambda}_n(\xi-\nabla\phi)\,\mu(d\xi)=m\lambda p^\lambda_n(\nabla\phi).
	\end{align*}
	Here, we have used
	\[\max_{u\in\real}\left|\frac{u}{1+\lambda^2 u^2}\right|=\frac{1}{2\lambda}.\]
	On the other hand, we have for $x=0$
	\begin{align*}
		\frac{\partial p^\lambda_n(\nabla\cdot)}{\partial \phi(0)}
		&=\int 2m\lambda^2\sum_{x\in\Lambda_n}\frac{\phi(x)-\phi(0)-\phi^{\xi,0}(x)}{1+\lambda^2(\phi(x)-\phi(0)-\phi^{\xi,0}(x))^2}
		\Phi^{\lambda}_n(\xi-\nabla\phi)\,\mu(d\xi)
	\end{align*}
	and
	\begin{align*}
		\left|\frac{\partial p^\lambda_n(\nabla\cdot)}{\partial \phi(0)}\right|
		&\le m\lambda|\Lambda_n|\int \Phi^{\lambda}_n(\xi-\nabla\phi)\,\mu(d\xi)=m\lambda|\Lambda_n|p^\lambda_n(\nabla\phi).
	\end{align*}
	We conclude that $(p^\lambda_n)^{-1}\partial p^\lambda_n(\nabla\cdot)/\partial \phi(x)$ is square-integrable.

	Next, let us verify that
	$(q_n)^{-1}\partial q_n(\nabla\cdot)/\partial \phi(x)$ is also square-integrable.
	Note that
	\[U_x^{\Lambda_n}(\nabla\phi)=q_n^{-1}\frac{\partial q_n(\nabla\cdot)}{\partial \phi(x)},\]
	and
	\begin{align*}
		\int &U_x^{\Lambda_n}(\nabla\phi)^2p^\lambda_n(\nabla\phi)\nu_{n,p}(d\phi) \\
		&\le 2\iint U_x^{\Lambda_n}(\xi)^2\Phi^{\lambda}_n(\nabla\phi-\xi)\nu_{n,p}(d\phi)\mu(d\xi) \\
		&\qquad {}+2\iint \left(U_x^{\Lambda_n}(\xi)-U_x^{\Lambda_n}(\nabla\phi)\right)^2\Phi^{\lambda}_n(\nabla\phi-\xi)\nu_{n,p}(d\phi)\mu(d\xi) \\
		&\le 2\int b_x(\xi,n)^2\mu(d\xi) \\
		&\qquad {}+K\iint \sum_{b\in\Lambda_n^*:x_b=x}\left|\xi(b)-\nabla\phi(b)\right|^2\Phi^{\lambda}_n(\nabla\phi-\xi)\nu_{n,p}(d\phi)\mu(d\xi)
	\end{align*}
	for some $K>0$ from Lipschitz continuity of $V'$. 
	It is easy to see that the first term is finite by using temperedness of $\mu$.
	We can obtain that the second term is also finite since we have
	\begin{align}\label{eq2.7c}
		\iint &\left|\xi(b)-\nabla\phi(b)\right|^2\Phi^{\lambda}_n(\nabla\phi-\xi)\nu_{n,p}(d\phi)\mu(d\xi) \\
		&\le 2\iint\left|\phi^{\xi,0}(x_b)-\phi(x_b)+\phi(0)\right|^2
		\Phi^{\lambda}_n(\nabla\phi-\xi)\nu_{n,p}(d\phi)\mu(d\xi) \nonumber \\
		& \qquad{}+2\iint \left|\phi^{\xi,0}(y_b)-\phi(y_b)+\phi(0)\right|^2
		\Phi^{\lambda}_n(\nabla\phi-\xi)\nu_{n,p}(d\phi)\mu(d\xi). \nonumber
	\end{align}
	and
	\begin{equation}\label{eq2.8c}
	\int_{\real}(x-a)^2\Phi_\lambda(x-a)\,dx\le C\lambda^{-2}
	\end{equation}
	by the definition of $\Phi$.
\end{proof}

Using $F(n)$, we can now bound the right hand side of \eqref{eq2.3c}:
\begin{lemma}
	Assume that the function $f$ satisfies $0\le uf'(u)\le 1$ for every $u>0$.
	We then have bounds for $R^\lambda_{1,x}(n,f)$ and $R^\lambda_{2,x}(n,f)$ in \eqref{eq2.3c} as follows:
	\begin{gather}
		\left|R^\lambda_{1,x}(n,f)\right|\le K_1 C_x(n)^{1/2}F^\lambda_x(n)^{1/2} \label{eq2.5c}\\
		\left|R^\lambda_{2,x}(n,f)\right|\le K_2 \lambda^{-1}F^\lambda_x(n)^{1/2} \label{eq2.6c}
	\end{gather}
	with some constants $K_1,K_2>0$ independent in $n$ and $\lambda$, where $C_x(n)$ is
	defined by
	\begin{gather*}
		C_x(n)=\sum_{b\in(\integer^d)^*\smallsetminus\Lambda^*:x_b=x}\int c^2_b(\eta,n,\mu)
		\mu(d\eta), \\
		c_b(\eta,n,\mu)=\int V'(\eta(b))\mu(d\eta|\mathscr{F}_{(\Lambda_n^*)^\complement})(\eta).
	\end{gather*}
\end{lemma}
\begin{proof}
We at first obtain 
\begin{align*}
\iint &f'\left(\frac{p^\lambda_n}{q_n}\right)^2\left(\frac{\partial}{\partial\psi(x)}
\left(\frac{p^\lambda_n}{q_n}\right)\right)^2
\Psi^\lambda_n(\eta,\nabla\psi)\nu_{\Lambda_n,p}(d\psi)\mu(d\eta) \\
&\le 
\iint \left(\frac{p^\lambda_n}{q_n}\right)^{-1}\left(\frac{\partial}{\partial\psi(x)}
\left(\frac{p^\lambda_n}{q_n}\right)\right)^2
q_n(\nabla\psi)\nu_{\Lambda_n,p}(d\psi)=F^\lambda_{x}(n)
\end{align*}
by the assumption on $f$. We therefore get
\begin{align*}
\left|R^\lambda_{1,x}(n,f)\right|
&\le F^\lambda_{x}(n)^{1/2}
\left(\iint \left(U_x(\eta)-U^\Lambda_x(\eta)\right)^2
	\Psi^\lambda_n(\eta,\nabla\psi)\nu_{\Lambda_n,p}(d\psi)\mu(d\eta)
\right)^{1/2} \\
&\le K F^\lambda_{x}(n)^{1/2}\left(\int\sum_{b\in(\integer^d)^*\smallsetminus\Lambda^*:x_b=x}c^2_b(\eta,n,\mu)
\mu(d\eta)
\right)^{1/2}
\end{align*}
for some constant $K_1>0$, which shows \eqref{eq2.5c}.
We note that $R^\lambda_{1,x}(n)$ is equal to zero if $x$ is not on the boundary of $\Lambda_n$.
Also for \eqref{eq2.6c}, we obtain
\begin{align*}
\left|R^\lambda_{2,x}(n,f)\right|
&\le F^\lambda_{x}(n)^{1/2}
\left(\iint 
	\left(U_x^\Lambda(\eta)-U_x^\Lambda(\nabla\psi)\right)^2
	\Phi^{\lambda}_n(\eta-\nabla\psi)\nu_{\Lambda_n,p}(d\psi)\mu(d\eta)\right)^{1/2}\\
	&\le K_2\lambda^{-1}F^\lambda_{x}(n)^{1/2}
\end{align*}
for some constant $K_2>0$ by applying \eqref{eq2.7c} and \eqref{eq2.8c} again.
\end{proof}

Summarizing above and applying Schwarz's inequality, we obtain 
\begin{align*}
\sum_{x\in\Lambda_n}F^\lambda_x(n,f)&\le \sum_{x\in\Lambda_n}(F^\lambda_x(n))^{1/2}(K_1C_x(n)^{1/2}+K_2\lambda^{-1} \\
& \le \frac{1}{2}\sum_{x\in\Lambda_n}F^\lambda_x(n)+K_1^2\sum_{x\in\Lambda_n}C_x(n)
+K_2^2\lambda^{-2}.
\end{align*}
By taking limit $f(u)$ to $\log u$ with keeping $0\le u f'(u)\le 1$ and applying Fatou's lemma.
\begin{equation}\label{eq2.9c}
\sum_{x\in\Lambda_n}F^\lambda_x(n)\le 2K_1^2 \sum_{x\in\Lambda_n}C_x(n)
+2K_2^2\lambda^{-2}\left|\Lambda_n\right|.
\end{equation}
Here, using Jensen's inequality and shift-invariance and temperedness
of $\mu$, we get
\[\int c_{b}(\xi,n,\mu)^2 \mu(d\xi)\le K<\infty,\]
with a constant $K>$ independent of $n$ and $b$, and therefore get
\begin{equation}\label{eq2.10c}
	\sum_{x\in\Lambda_n}C_x(n)\le 2dK |\Lambda_{n}\smallsetminus\Lambda_{n-1}|.
\end{equation}
Summarizing \eqref{eq2.9c} and \eqref{eq2.10c}, we get
\begin{equation}\label{eq2.11c}
	\sum_{x\in\Lambda_n}F^\lambda_x(n)\le 4dK_1^2 K |\Lambda_{n}\smallsetminus\Lambda_{n-1}|+2K_2^2 \lambda^{-2}
	|\Lambda_{n}|.
\end{equation}
%
We  note that the left hand side of \eqref{eq2.11c} coincides with the entropy
production rate, that is, the identity
\begin{align*}
	\sum_{x\in\Lambda_n}F^\lambda_x(n)&=\mathscr{E}^{\Lambda_n,\mathrm{f}}(\sqrt{r^\lambda_n},\sqrt{r^\lambda_n}) \\
	&=\sup\left\{\int_{\mathcal{X}_{\Lambda_n^*}}\frac{-\mathscr{L}^{\Lambda,\mathrm{f}}u}{u}d\mu^\lambda_n;\,
	u\in C_{b}^2(\mathcal{X}),\,\mathscr{F}_{\Lambda^*}\text{-measurable},\,u\ge 1\right\} 
\end{align*}
holds, where the probability measure $\mu^\lambda_n$ on $\Lambda_n^*$
is defined by $d\mu^\lambda_n=r^\lambda_nd\mu^{\Lambda_n,\mathrm{f}}$.
We note that the core of Dirichlet form $\mathscr{E}^{\Lambda_n,\mathrm{f}}$
is the family of smooth $\mathscr{F}_{\Lambda_n^*}$-measurable functions.

For $\ell\in\natural$, let us take $\tilde\Lambda\subset\Lambda_n$ by
\[\tilde\Lambda=\bigcup_{x\in((2\ell+3)\integer)^d;\Lambda_\ell(x)\subset\Lambda_{n-1}}\Lambda_\ell(x),\]
where $\Lambda_\ell(x)=\Lambda_\ell+x$. Because boxes $\Lambda_\ell(x)$ appearing above are disjoint, we get
\[
	\sum_{x\in((2\ell+3)\integer)^d;\Lambda_\ell(x)\subset\Lambda_{n-1}}	\sum_{y\in\Lambda_\ell(x)}F_y^{\lambda}(n)\le \sum_{x\in\Lambda_n}F_x^\lambda(n).
\]
On the other hand, we have
\[
\sum_{y\in\Lambda_\ell(x)}F_x^{\lambda}(n)
=I^{\Lambda_\ell(x)}(\mu_n^\lambda),
\]
where the right hand side is the entropy production rate defined by
\[
	I^{\Lambda}(\tilde\mu)
	:=\sup\left\{\int\frac{-\mathscr{L}^{\Lambda}u}{u}d\tilde\mu;\,
	u\in C_{b}^2(\mathcal{X}),\,\mathscr{F}_{\overline{\Lambda^*}}\text{-measurable},\,u\ge 1\right\}
\]
for a finite $\Lambda\subset\integer^d$ and $\tilde\mu$ on $\mathcal{X}$
or $\mathcal{X}_{\overline{\Lambda^*_n}}$ with $n$ large enough.
Repeating the argument as in
the proof of Lemma~4.2 in \cite{FS97}, we obtain 
the Gibbsian property of $\mu$.

\section*{Aknowledgement}
The first author thanks RIMS at Kyoto University
and The University of Tokyo for the kind hospitality and the financial support.
The second author was partially supported by the grant
for scientific research by College of Science and Technology, Nihon University.

\bibliographystyle{amsplain}
\bibliography{hydro}

\providecommand{\bysame}{\leavevmode\hbox to3em{\hrulefill}\thinspace}
\providecommand{\MR}{\relax\ifhmode\unskip\space\fi MR }
\providecommand{\MRhref}[2]{%
  \href{http://www.ams.org/mathscinet-getitem?mr=#1}{#2}
}
\providecommand{\href}[2]{#2}
\begin{thebibliography}{10}

\bibitem{BK07}
M.~Biskup and R.~Koteck{\'y}, \emph{Phase coexistence of gradient {G}ibbs
  states}, Probab. Theory Related Fields \textbf{139} (2007), no.~1-2, 1--39.
  \MR{2322690 (2008d:82024)}

\bibitem{CD08}
C.~Cotar and J.-D. Deuschel, \emph{Decay of covariances, uniqueness of ergodic
  component and scaling limit for a class of {$\nabla\phi$} systems with
  non-convex potential}, Ann. Inst. Henri Poincar\'e Probab. Stat. \textbf{48}
  (2012), no.~3, 819--853. \MR{2976565}

\bibitem{CDM08}
C.~Cotar, J.-D. Deuschel, and S.~M\"uller, \emph{Strict convexity of the free
  energy for a class of non-convex gradient models}, Commun. Math. Phys.
  \textbf{286} (2009), 359--376.

\bibitem{DGI00}
J.-D. Deuschel, G.~Giacomin, and D.~Ioffe, \emph{Large deviations and
  concentration properties for $\nabla\varphi$ interface models}, Probab.
  Theory Relat. Fields \textbf{117} (2000), 49--111.

\bibitem{F82a}
J.~Fritz, \emph{Stationary measures of stochastic gradient systems, infinite
  lattice models}, Z. Wahrsch. Verw. Gebiete \textbf{59} (1982), no.~4,
  479--490. \MR{MR656511 (83j:60108)}

\bibitem{F12}
T.~Funaki, \emph{Hydrodynamic limit for the $\nabla\varphi$ interface model via
  two-scale approach}, Probability in Complex Physical Systems: In Honour of
  Erwin Bolthausen and J\"{u}rgen G\"{a}rtner, Springer Proceedings in
  Mathematics, vol.~11, Springer, Heidelberg, 2012.

\bibitem{FS97}
T.~Funaki and H.~Spohn, \emph{Motion by mean curvature from the
  {Ginzburg-Landau} $\nabla\phi$ interface model}, Commun. Math. Phys.
  \textbf{185} (1997), 1--36.

\bibitem{GOVW09}
N.~Grunewald, F.~Otto, C.~Villani, and M.G. Westdickenberg, \emph{A two-scale
  approach to logarithmic {S}obolev inequalities and the hydrodynamic limit},
  Ann. Inst. Henri Poincar\'e Probab. Stat. \textbf{45} (2009), no.~2,
  302--351. \MR{2521405 (2010c:60293)}

\bibitem{HS77}
R.~A. Holley and D.~W. Stroock, \emph{In one and two dimensions, every
  stationary measure for a stochastic {I}sing model is a {G}ibbs state}, Comm.
  Math. Phys. \textbf{55} (1977), no.~1, 37--45. \MR{0451455 (56 \#9741)}

\bibitem{HS81}
\bysame, \emph{Diffusions on an infinite-dimensional torus}, J. Funct. Anal.
  \textbf{42} (1981), no.~1, 29--63. \MR{620579 (82k:60152)}

\bibitem{KL14}
R.~Koteck{\'y} and S.~Luckhaus, \emph{Nonlinear elastic free energies and
  gradient {Y}oung-{G}ibbs measures}, Comm. Math. Phys. \textbf{326} (2014),
  no.~3, 887--917. \MR{3173410}

\bibitem{N03}
T.~Nishikawa, \emph{Hydrodynamic limit for the {Ginzburg-Landau} $\nabla\phi$
  interface model with boundary conditions}, Probab. Theory Relat. Fields
  \textbf{127} (2003), 205--227.

\bibitem{BRW04}
F.-Y.~Wang V.I.~Bogachev, M.~R{\"o}ckner, \emph{Invariance implies {G}ibbsian:
  some new results}, Comm. Math. Phys. \textbf{248} (2004), no.~2, 335--355.
  \MR{2073138 (2005g:58064)}

\end{thebibliography}
\end{document}